\documentclass[11pt,letterpaper]{amsart}
\usepackage{amssymb}
\usepackage{color}
\usepackage[linkcolor=blue,citecolor=blue, colorlinks=true,backref=true,bookmarks=true]{hyperref}

\theoremstyle{plain}
\newtheorem{theorem}{Theorem}[section]
\newtheorem{corollary}[theorem]{Corollary}

\newtheorem{lemma}[theorem]{Lemma}
\newtheorem{proposition}[theorem]{Proposition}
\newtheorem{definition}[theorem]{Definition}

\theoremstyle{definition}
\newtheorem{remark}[theorem]{Remark}

\def\R{\mathbb{R}}
\def\S{\mathbb{S}}
\def\H{\mathbb{H}}
\def\Z{\mathbb{Z}}
\def\C{\mathbb{C}}

\def\N{\mathbb{N}}

\def\H{\mathbb{H}}

\newcommand{\PP}{\mathcal{P}}

\numberwithin{equation}{section}

\title{Improved higher-order Sobolev inequalities on CR sphere}

\begin{document}
\author{Zetian Yan}
\address{Department of Mathematics \\ Penn State University \\ University Park \\ PA 16802 \\ USA}
\email{zxy5156@psu.edu}
\begin{abstract}
We improve higher-order CR Sobolev inequalities on $S^{2n+1}$ under the vanishing of higher order moments of the volume element. As an application,
we give a new and direct proof of the classification of minimizers of the CR invariant higher-order Sobolev inequalities. In the same spirit, we prove almost sharp Sobolev inequalities for GJMS operators to general CR manifolds, and obtain the existence of minimizers in $C^{2k}(N)$ of higher-order CR Yamabe-type problems when $Y_k(N)<Y_k(\H^n)$. 
\end{abstract}
\maketitle

\section{Introduction}
The main aim of this note is to improve CR Sobolev inequalities on $S^{2n+1}$ under the vanishing of higher order moments of the area elements; see Theorem \ref{improve} for a precise statement. As a direct application, we simplify Frank and Lieb's proof \cite{FL10} of the existence of minimizers of the sharp Sobolev inequalities for the embeddings $S^{k,2}(\H^n)\hookrightarrow L^{\frac{2n+2}{n+1-k}}(\H^n)$. Moreover, using the Frank--Lieb argument developed in \cite{Case19}, we give a simple and direct classification of optimizers. In particular, our proof avoids complicated computations, partially addressing an open problem of Frank and Lieb \cite[pg.352]{FL10}.

Our derivation of Theorem \ref{improve} is motivated by recent work in \cite{CH22} and \cite{HW21}, where Aubin's Moser--Trudinger--Onofri inequality on $S^n$ and Sobolev inequalities on $S^n$ are improved under the vanishing of higher order moments of the volume element, respectively. The key observation in \cite{HW21} is that, if it is not true, there exists a sequence of functions $\{F_i\}$ in $W^{1,p}(S^n)$ such that
\begin{displaymath}
\|\nabla F_i\|^p_p\leqslant \frac{1}{\alpha}, \quad \|F_i\|_{p^*}=1, 
\end{displaymath}
and $\{F_i\}$ converges to zero weakly in $L^p(S^n)$, where $p^*=\frac{np}{n-p}$ and $\alpha$ is the leading coefficient of the improved Sobolev inequality \cite[Theorem 1.1]{HW21}. Combining this with the concentration compactness principle \cite[Lemma 4.1]{Lions85}, this contradicts the choice of $\alpha$. We will apply this idea to prove Theorem \ref{improve} in Section 3. The crucial point of the proof is the concentration compactness principle Lemma \ref{ccp} on general CR manifolds. 

Let $P_k^{\theta_c}$, $k\in\N$ with $k<n+1$, be the CR GJMS operators associated with the canonical contact form $\theta_c$ on $\H^n$. For any $f\in S^{k,2}(\H^n)$, where $S^{k,2}(\H^n)$ is the Folland--Stein space introduced in \cite{FS74}, $p=\frac{2Q}{Q-2k}$, the sharp Hardy-Littlewood-Sobolev inequalities are given by 
\begin{equation}\label{optimal}
\begin{split}
    \left(\int_{\H^n}|f(u)|^pdu\right)^{\frac{2}{p}}&\leqslant C_{n,2k} \int_{\H^n}\bar{f} P_k^{\theta_c} f du,\\
    C_{n,2k}&=\left(\frac{1}{4\pi}\right)^{k}\frac{\Gamma^2(\frac{n+1-k}{2})}{\Gamma^2(\frac{n+1+k}{2})},
    \end{split}
\end{equation}
where $du=\theta_c\wedge (d\theta_c)^n$ is the volume form associated with $\theta_c$ on $\H^n$; see \cite{FL10} for more details. \\
~\\
Equivalently, via the Cayley transform $\mathcal{C}:\H^n\to S^{2n+1}\backslash (0,\cdots, -1)$ given by 
\begin{displaymath}
    \mathcal{C}(z,t)=\left(\frac{2z}{1+|z|^2+it}, \frac{1-|z|^2-it}{1+|z|^2+it}\right),
\end{displaymath}
we have for any $F\in S^{k,2}(S^{2n+1})$,
\begin{equation}\label{sharpSo}
\left(\int_{S^{2n+1}}|F|^p d\xi\right)^{\frac{2}{p}}\leqslant C_{n,2k}\int_{S^{2n+1}}\bar{F} P_k^{\theta_0} F d\xi,
\end{equation}
where $d\xi=\theta_0\wedge (d\theta_0)^n$ is the volume form associated with the canonical contact form $\theta_0$ on $S^{2n+1}$.

Recall that a homogeneous polynomial on $\C^{n+1}$ with bidegree $(w,w')$ is a polynomial $g$ satisfying for any $\lambda\in \C$, $z\in \C^{n+1}$,
\begin{displaymath}
    g(\lambda z)=\lambda^w \bar{\lambda}^{w'}g(z).
\end{displaymath}
For a pair of nonnegative integers $(w, w')$, we denote
\begin{center}
$\PP_{w,w'}:=\{$all homogeneous polynomials on $\C^{n+1}$ with bidegree $(w,w')\}$,

$\widetilde{\PP}_{w,w'}:=\bigcup\limits_{(\tilde{w},\tilde{w}')}\PP_{\tilde{w},\tilde{w}'}$, \quad for all $\tilde{w}\leqslant w$, $\tilde{w}'\leqslant w'$,

$\overline{\PP}_{w,w'}:=\left\{g\in \widetilde{\PP}_{w,w'}; \int_{S^{2n+1}}g d\xi=0 \right\}$.
\end{center}
For $m\in \N$ and $0\leqslant \theta\leqslant 1$, as in \cite{HW21} we define
\begin{center}
$\mathcal{M}^c_{w,w'}(S^{2n+1}):=\{\nu$; $\nu$ is a probability measure on $S^{2n+1}$ supported on countably many points such that $\int_{S^{2n+1}}gd\nu=0$, for all $g\in \overline{\PP}_{w,w'}\}$
\end{center}
and 
\begin{displaymath}
    \Theta({w,w'};\theta,2n+1):=\inf\left\{\sum_{i}\nu^{\theta}_i; \nu\in \mathcal{M}^c_{w,w'}(S^{2n+1}), \mathrm{supp}(\nu)=\{x_i\}\subset S^{2n+1}, \nu_i=\nu(x_i)\right\}.
\end{displaymath}

Our first result gives improved CR Sobolev inequalities under the assumption that higher moments vanish.
\begin{theorem}\label{improve}
Denote $Q=2n+2$, $p=\frac{2Q}{Q-2k}$, $k\in\N$ with $k<n+1$. Then for any $\epsilon>0$, and $F\in S^{k,2}(S^{2n+1})$ with
\begin{equation}\label{constraint}
    \int_{S^{2n+1}}g|F|^p d\xi=0
\end{equation}
for all $g\in \overline{\PP}_{w,w'}$, we have
\begin{equation}\label{asi}
\begin{split}
   \left(\int_{S^{2n+1}}|F|^p d\xi\right)^{\frac{2}{p}}&\leqslant \left(\frac{{C}_{n,2k}}{\Theta({w,w'};\frac{Q-2k}{Q},2n+1)}+\epsilon\right) \int_{S^{2n+1}}\bar{F} P_k^{\theta_0} F d\xi\\
   &+C(\epsilon)\int_{S^{2n+1}}|F|^2 d\xi,
\end{split}
\end{equation}
where $C(\epsilon)$ is a constant depending on $\epsilon$.
\end{theorem}
When $w+w'=1$, the condition $\int_{S^{2n+1}}gd\nu=0$ for all $g\in \overline{\PP}_{w,w'}$ is equivalent to the balanced conditions
\begin{equation}\label{balanced}
    \int_{S^{2n+1}}\xi_i d\nu=0, \quad i=1,\cdots, 2n+2,
\end{equation}
where $\{\xi_i\}_{i=1}^{2n+2}$ are coordinate functions on $\R^{2n+2}$. Therefore, the exact value of $\Theta({w,w'}; \theta,2n+1)$ is $2^{1-\theta}$ given in \cite{HW21}. We have the following corollary for balanced functions on $S^{2n+1}$.
\begin{corollary}\label{impro}
Denote $Q=2n+2$, $p=\frac{2Q}{Q-2k}$, $k\in\N$ with $k<n+1$. Then for any $\epsilon>0$, and $F\in S^{k,2}(S^{2n+1})$ with
\begin{equation}
    \int_{S^{2n+1}}\xi_i|F|^p d\xi=0, \quad i=1,\cdots, 2n+2,
\end{equation}
we have
\begin{equation}
\left(\int_{S^{2n+1}}|F|^p d\xi\right)^{\frac{2}{p}}\leqslant \left(\frac{{C}_{n,2k}}{2^{\frac{k}{n+1}}}+\epsilon\right) \int_{S^{2n+1}}\bar{F} P_k^{\theta_0} F d\xi+C(\epsilon)\int_{S^{2n+1}}|F|^2 d\xi,
\end{equation}
where $C(\epsilon)$ is a constant depending on $\epsilon$.
\end{corollary}
The author recently learned that P. Ho obtained a similar result independently. In his preprint, he gets improved inequalities for the embedding $S^{1,r}(S^{2n+1})\hookrightarrow L^{\frac{Qr}{Q-r}}(S^{2n+1})$ with $0<r<Q$.  
\begin{remark}
On the standard unit sphere $(S^{n}, g_0)$, it is known in \cite[Theorem 2.45]{Aubin76} that for all $\varphi\in H^{1,2}(S^{n})$ satisfying 
\begin{equation}
    \int_{S^{n}}\xi_i|\varphi|^{\frac{2n}{n-2}} d\xi=0, \quad i=1,\cdots, n+1,
\end{equation}
we have
\begin{equation}\label{Aubin}
    \|\varphi\|^2_{\frac{2n}{n-2}}\leqslant \left(\frac{K_{n,2}}{2^{\frac{2}{n}}}+\epsilon\right)\|\nabla \varphi\|^2_2+C(\epsilon)\|\varphi\|^2_2,
\end{equation}
where $K_{n,2}$ is the sharp constant of the Sobolev inequality on $\R^{n+1}$ and $C(\epsilon)$ is a constant depending on $\epsilon$. Corollary \ref{impro} is a CR analogue of Aubin's estimate which works for all orders $2k<Q$.
\end{remark}

\begin{remark} 
If we replace $\PP_{w,w'}$ by $\PP_l$ defined in \cite{HW21} 
\begin{center}
$\PP_l:=\{$all polynomials on $\R^{2n+2}$ with degree at most $l\}$
\end{center}
with $w+w'=l\in \N$, we can adapt the proof of Theorem \ref{improve} to obtain 
\begin{equation}\label{asi2}
\begin{split}
\left(\int_{S^{2n+1}}|F|^p d\xi\right)^{\frac{2}{p}}&\leqslant \left(\frac{{C}_{n,2k}}{\Theta(l;\frac{Q-2k}{Q},2n+1)}+\epsilon\right) \int_{S^{2n+1}}\bar{F} P_k^{\theta_0} F d\xi\\
&+C(\epsilon)\int_{S^{2n+1}}|F|^2 d\xi,
\end{split}
\end{equation}
under the restriction
\begin{equation}\label{constraint2}
    \int_{S^{2n+1}}g|F|^p d\xi= 0
\end{equation}
for all $g\in \overline{\PP}_l$, where $\Theta(l;\frac{Q-2k}{Q},2n+1)$ and $\overline{\PP}_l$ are analogues associated with $\PP_l$. Since $\PP_{w,w'}$ is a proper subspace of $\PP_l$ when $l\geqslant 2$, condition (\ref{constraint2}) is more restrictive than (\ref{constraint}). Consequently, we have
\begin{displaymath}
    \frac{{C}_{n,2k}}{\Theta(l;\frac{Q-2k}{Q},2n+1)}\leqslant \frac{{C}_{n,2k}}{\Theta(w,w';\frac{Q-2k}{Q},2n+1)}.
\end{displaymath}
\end{remark}

\begin{remark}
	Some known exact values of $\Theta(l;\theta,n-1)$ are $\Theta(1;\theta,n-1)=2^{1-\theta}$ and $\Theta(2;\theta,n-1)=(n+1)^{1-\theta}$ by Hang--Wang \cite{HW21}, and $\Theta(3;\theta,n-1)=(2n)^{1-\theta}$ by Putterman \cite[Theorem 5.1]{Putterman20}, whose method is related to the idea of deriving cubature formulas, such as the technique of reproducing kernels on spheres, etc. In Proposition \ref{value}, we compute $\Theta (1,1; \theta, 2n+1)$ following the idea in \cite{HW21}. Indeed, it has been proved in \cite[Proposition 3.1]{Putterman20} that $\Theta(l;\theta,n-1)$ can only be achieved by a Dirac probability measure supported on finitely many points. This directly implies that the infimum for $\Theta(l;\theta,n-1)$ is a minimum, which follows from \cite[Corollary 3.2]{Putterman20} or
	the proof of Proposition \ref{value} below implicitly.  
\end{remark}

As a direct application, we give a new proof of sharp CR Sobolev inequalities on $S^{2n+1}$.

\begin{theorem}\label{class}
Let $(S^{2n+1}, \theta_0)$ be the CR unit sphere with the volume element $d\xi=\theta_0\wedge (d\theta_0)^n$. Denote $Q=2n+2$ and $k\in\N$ with $k<n+1$. Then $u$ is a positive minimizer of (\ref{sharpSo}) if and only if
\begin{displaymath}
u(\eta)=\frac{C}{|1-\xi \cdot \bar{\eta}|^{\frac{Q-2k}{2}}}, \quad \eta\in S^{2n+1}
\end{displaymath}
for some $C\in \C, \xi\in \C^{n+1}, |\xi|<1$.
\end{theorem}
The key ideas in the proof of Theorem \ref{class} are the following. After using the CR automorphism group $\mathcal{A}(S^{2n+1})$ (see, for example, \cite{BFM07}), without loss of generality, we may assume that any minimizing sequence $\{F_i\}$ in $S^{k,2}(S^{2n+1})$ is balanced. By the Lieb trick in \cite[Lemma 2.6]{Lieb83} and the improved Sobolev inequalities in Corollary \ref{impro}, up to a subsequence, we show that any balanced minimizing sequence $\{F_i\}$ converges strongly to an optimizer $F\in S^{k,2}(S^{2n+1})$ of (\ref{sharpSo}). Next, we follow the Frank--Lieb argument developed in \cite{Case19} to classify the optimizers. The Frank--Lieb argument involves three elements. First, the assumption of a local miminizer implies, via the second variation, a nice spectral estimate; see Proposition \ref{spectral} for a precise statement. Second, CR covariance implies that one can assume that the optimizer $F$ satisfies the balanced condition (\ref{balanced}), and in particular use first spherical harmonics as test functions in the previous spectral estimate. Third, one deduces that a balanced positive local minimizer is constant by computing a particular commutator formula involving the relevant CR GJMS operator and a first spherical harmonic. The commutator formula needed for Theorem \ref{class} is given in Theorem \ref{commutator}.

In the same spirit as Theorem \ref{improve}, we prove Aubin-type almost Sobolev inequalities involving GJMS operators on general CR manifolds. 

\begin{theorem}\label{almost2}
Let $(N, T^{1,0}N,\theta)$ be a compact strictly pseudoconvex manifold of dimension $2n+1$. Then for any $F\in S^{k,2}(N)$, $p=\frac{2Q}{Q-2k}$, we have
\begin{equation}\label{Sobolev2}
    \left(\int_{N}|F|^p d\zeta\right)^{\frac{2}{p}}\leqslant \left(C_{n,2k}+\epsilon\right)\int_{N}\bar{F} \tilde{P}^{\theta}_k F d\zeta+C(\epsilon)\int_N |F|^2 d\zeta,
\end{equation}
where $d\zeta=\theta\wedge d\theta^n$, $C_{n,2k}$ is the optimal constant on $\H^n$ introduced in (\ref{optimal}), and $\tilde{P}^{\theta}_k$ is the principle part of $P^{\theta}_k$ given in (\ref{prin}).
\end{theorem}

In the Yamabe problem, it is known \cite{LP87} that $Y(M^n)\leqslant Y(S^n)$, where the equality holds if and only if $(M,g)$ is conformal to the standard sphere $(S^n, g_c)$. When $Y(M^n)<Y(S^n)$, the Yamabe constant $Y(M^n)$ can be attained by a positive minimizer; the proof of this relies on Aubin's almost sharp inequality \cite[Theorem 2.3]{LP87}. In the CR Yamabe problem, Jerison and Lee proved the existence of minimizers in \cite[Theorem 3.4]{JL87}. However, they did not have an almost sharp Sobolev inequality on general CR manifolds. To overcome this technical difficulty, they proved the existence by contradiction by lifting a divergent sequence of functions to $\H^n$ in CR normal coordinates.

With help of Theorem \ref{almost2}, we obtain the existence of minimizers in $C^{2k}(N)$ of higher-order CR Yamabe-type problems when $Y_k(N)<Y_k(\H^n)$. In particular, this provides a new, direct proof of Theorem 3.4(c) in \cite{JL87}. However, when $k\geqslant 2$, it is difficult to know whether the minimizer is positive. 
\begin{theorem}\label{highCR}
Let $(N^{2n+1}, T^{1,0}N)$ be a compact strictly pseudoconvex CR manifold with contact form $\theta$. Denote $Q=2n+2$ and $k\in\N$ with $k<n+1$. We consider the extremal problem
\begin{equation}\label{constant}
    Y_k(N)=\inf \left\{A^{\theta}_k(u);u\in S^{k,2}(N;\R),  B^{\theta}_k(u)=1 \right\},
\end{equation}
where 
\begin{align}\label{AB}
    A^{\theta}_k(u)& =\int_N u P_k^{\theta}(u) \theta\wedge d\theta^n,\\
    B^{\theta}_k(u)& =\int_N |u|^p \theta\wedge d\theta^n, \quad p=\frac{2Q}{Q-2k}.
\end{align} 
If $Y_k(N)<Y_k(\H^n)=C_{n,2k}^{-1}$ and $\{F_i\}$ is a minimizing sequence of $Y_k(N)$ with $B^{\theta}_k(F_i)=1$, then, after passing to a subsequence, we can find $F\in S^{k,2}(N;\R)$ such that $F_i\to F$ strongly in $S^{k,2}(N)$. In particular, $F$ is a minimizer of $Y_k(N)$ in $C^{2k}(N)$.
\end{theorem}

In the CR Yamabe problem, the infimum may not be achieved when the CR Yamabe constant $Y_1(N)=Y_1(\H^n)$. Recently, in \cite{CMY19}, authors exhibit examples of compact three-dimensional CR manifolds of positive Webster class for which the pseudohermitian mass \cite{CMY17} is negative, and for which $Y_1(N)$ is not attained. This phenomenon is in striking contrast to the Riemannian case.

This article is organized as follows.

In Section 2, we review some basic concepts on CR manifolds, ambient spaces and CR GJMS operators. In Section 3, we prove Theorem \ref{improve} and compute $\Theta (1,1; \theta, 2n+1)$ following the idea in \cite{HW21}. Section 4 deals with the commutator identity and spectral estimate needed for the Frank--Lieb argument. As a direct application, we give a
new proof of sharp CR Sobolev inequalities on $S^{2n+1}$. In Section 5, we provide the proof of Theorem \ref{almost2}. In Section 6, using Lemma \ref{ccp}, we prove Theorem \ref{highCR}.

{\bf{Acknowledgements.}} I would like to express my deep gratitude to my advisor, Dr. Jeffrey S. Case, for the patient guidance and useful critiques of this work. I also would like to thank Pak Tung Ho for making his work available and Nan Wu for his encouragement.

\section{Preliminaries}
\subsection{CR manifolds}
Let $N$ be a smooth $(2n+1)$-dimensional manifold without boundary. A CR structure on $N$ is a complex $n$-dimensional subbundle $T^{1,0}N$ of the complexified tangent bundle $\C TN$ such that
\begin{itemize}
    \item $T^{1,0}N\cap T^{0,1}N=\{0\}$, where $T^{0,1}N=\overline{T^{1,0}N}$;
    \item $[\Gamma(T^{1,0}N),\Gamma(T^{1,0}N)]\subset \Gamma(T^{1,0}N).$
\end{itemize}
For example, if $N$ is a real hypersurface in a complex manifold $X$, then $N$ has the natural CR structure
\begin{displaymath}
    T^{1,0}N=(TN\otimes \C)\cap T^{1,0}X.
\end{displaymath}
We set $HN=\mathrm{Re}$ $(T^{1,0}N\oplus T^{0,1}N)$. In the following, assume that there exists a global nonvanishing $1$-form $\theta$ that annihilates $HN$. The {\textit{Levi form}} $L_{\theta}$ with respect to $\theta$ is the Hermitian form on $T^{1,0}N$ defined by
\begin{displaymath}
    L_{\theta}(V,\overline{W})=\langle -id\theta, V\wedge \overline{W} \rangle, \quad V,W\in T^{1,0}N.
\end{displaymath}
We consider only {\textit{strictly pseudoconvex CR manifolds}}, i.e. CR manifolds that have a positive definite Levi form for a suitable choice of $\theta$. In this case, $\theta$ defines a contact structure on $N$, and we call $\theta$ a contact form. We denote by $T$ the {\textit{Reeb vector field}} with respect to $\theta$; that is, the unique vector field satisfying
\begin{displaymath}
    \theta(T)=1, \quad \iota_T d\theta=0.
\end{displaymath}
Then the tangent bundle has the decomposition $TN=HN\oplus \R T$.
Take a local frame $(Z_{\alpha})$ of $T^{1,0}N$. Then we have a local frame $(T, Z_{\alpha}, Z_{\bar{\beta}}=\overline{Z_{\beta}})$ of $\C TN$, and take the dual frame $(\theta, \theta^{\alpha}, \theta^{\bar{\beta}})$. For this frame, the $2$-form $d\theta$ can be written
\begin{equation}\label{metric}
    d\theta=ih_{\alpha\bar{\beta}}\theta^{\alpha}\wedge \theta^{\bar{\beta}},
\end{equation}
where $(h_{\alpha\bar{\beta}})$ is a positive Hermitian matrix. In the following, we will use $h_{\alpha\bar{\beta}}$ and its inverse $h^{\alpha\bar{\beta}}$ to lower and raise indices of various tensors.

A $pseudohermitian$ $structure$ on $N$ is a CR structure together with a given contact form $\theta$. With this choice, $N$ is equipped with a natural volume form $\theta\wedge d\theta^n$. The inner product $L_{\theta}$ determines an isomorphism $HN \cong HN^{*}$, which in turn determines a dual form $L^{*}_{\theta}$ on $HN^{*}$, which extends naturally to $T^{*}N$. This defines a norm $|\omega|_{\theta}$ on real $1$-forms $\omega$, which satisfies
\begin{displaymath}
    |\omega|^2_{\theta}=L_{\theta}^{*}(\omega,\omega)=2\sum^{n}_{j=1} \left|\omega(Z_j)\right|^2,
\end{displaymath}
where $Z_1,\cdots, Z_n$ form an orthonormal basis for $T^{1,0}N$ with respect to the Levi form.

The $sub$-$Laplacian$ $operator$ $\Delta_b$ is defined on real functions $u\in C^{\infty}(N)$ by
\begin{equation}\label{subla}
    \int_N \left(\Delta_b u\right)v \theta\wedge d\theta^n=\int_N L^{*}_{\theta}(du,dv)\theta\wedge d\theta^n, \forall v\in C^{\infty}_{0}(N).
\end{equation}
Since evidently $|\theta|_{\theta}=0$, $L^{*}_{\theta}$ is degenerate on $T^{*}N$. Kohn's work \cite{Kohn65} shows that the operator $\Delta_b$ is subelliptic rather than elliptic. It is known \cite{Lee86} that $\Delta_b=\mathrm{Re} \Box_b$, where $\Box_b$ is the Kohn Laplacian acting on functions.

\subsection{Ambient space and ambient construction}
Let $X$ be an $(n+1)$-dimensional complex manifold and $\pi_{\mathcal{X}}:\mathcal{X}=K^{\times}_{X}\to X$ the total space of the canonical bundle of $X$ with the zero section removed. For $\mu\in \C^{\times}$, define the dilation $\boldsymbol{{\delta}}_{\mu}:\mathcal{X}\to \mathcal{X}$ by the scalar multiplication $\boldsymbol{{\delta}}_{\mu}(\varsigma)=\mu^{n+2}\varsigma$ for $\varsigma\in \mathcal{X}$. Denote by $Z$ the holomorphic vector field generating $\boldsymbol{{\delta}}_{\mu}$; that is, $\boldsymbol{{\delta}}_{\mu}=(d/d\mu)|_{\mu=1}\boldsymbol{{\delta}}_{\mu}^{*}$. A smooth function $\boldsymbol{f}$ on an open set of $\mathcal{X}$ is said to be {\textit{homogeneous of bidegree}} $(w,w')$ for $w, w'\in \C$ satisfying $w-w'\in\Z$ if $Z\boldsymbol{f}=w\boldsymbol{f}$ and $\overline{Z}\boldsymbol{f}=w'\boldsymbol{f}$. Without further comment, we will assume implicitly  that $w, w'\in \C$ and $w-w'\in\Z$ whenever we write $\tilde{\mathcal{E}}(w, w')$ for the sheaf of smooth homogeneous functions of bidegree $(w,w')$. To simplify notation, write $\tilde{\mathcal{E}}(w)=\tilde{\mathcal{E}}(w, w)$.

Let $N$ be a strictly pseudoconvex real hypersurface in $X$ and $\mathcal{N}=\pi_{\mathcal{X}}^{-1}(N)$. Then $\mathcal{N}$ is a strictly pseudoconvex real hypersurface in  $\mathcal{X}$. Denote by $\mathcal{E}(w, w')$ the sheaf of homogeneous functions of bidegree $(w,w')$ on $\mathcal{N}$. 

For a defining function ${\boldsymbol{\rho}}\in \tilde{\mathcal{E}}(1)$ of $\mathcal{N}$, the $(1,1)$-form $dd^c {\boldsymbol{\rho}}$ defines a Lorentz-K\"ahler metric ${\boldsymbol{g[\rho]}}$ in a neighborhood of $\mathcal{N}$, where $d^c=(i/2)(\overline{\partial}-\partial)$. We normalize ${\boldsymbol{\rho}}$ by a complex Monge-Amp\`ere equation. Take the tautological $(n+1,0)$-form ${\boldsymbol{\zeta}}$ on $K_{X}$. Then
\begin{displaymath}
    \mathrm{vol}_{\mathcal{X}}=i^{(n+2)^2}d{\boldsymbol{\zeta}} \wedge \overline{d {\boldsymbol{\zeta}}}
\end{displaymath}
gives a volume form on $\mathcal{X}$.
\begin{proposition}[{\cite[Proposition 2.2]{Hirachi17}}]
There exists a defining function ${\boldsymbol{\rho}}\in \tilde{\mathcal{E}}(1)$ of $\mathcal{N}$ such that
\begin{displaymath}
    \left(dd^c {\boldsymbol{\rho}}\right)^{n+2}=-\frac{(n+1)!}{n+2}\left(1+{\boldsymbol{\mathcal{O}}}{\boldsymbol{\rho}}^{n+2}\right)\mathrm{vol}_{\mathcal{X}},
\end{displaymath}
where ${\boldsymbol{\mathcal{O}}}\in \tilde{\mathcal{E}}(-n-2)$. Moreover, such a ${\boldsymbol{\rho}}$ is unique modulo $O\left({\boldsymbol{\rho}}^{n+3}\right)$, and ${\boldsymbol{\mathcal{O}}}$ modulo $O\left({\boldsymbol{\rho}}\right)$ is independent of the choice of ${\boldsymbol{\rho}}$.
\end{proposition}
We call such a ${\boldsymbol{\rho}}$ a {\textit{Fefferman defining function}} and Lorentz-K\"ahler metric ${\boldsymbol{g[\rho]}}$ with respect to ${\boldsymbol{\rho}}$ an {\textit{ambient metric}}. The function ${\boldsymbol{\mathcal{O}}}$ is called the {\textit{obstruction function}}.

Next, we recall CR invariant powers of the sub-Laplacian. Let ${\boldsymbol{\Delta}}$ be the Laplacian with respect to an ambient metric ${\boldsymbol{g[\rho]}}$. This operator maps $\tilde{\mathcal{E}}(w, w')$ to $\tilde{\mathcal{E}}(w-1, w'-1)$

\begin{lemma}[{\cite[Theorem 1.1]{GG05}}]\label{construction}
Let $(w,w')$ be such that $k=n+1+w+w'$ is a positive integer. Then for ${\boldsymbol{\tilde{f}}}\in \tilde{\mathcal{E}}(w, w')$,
\begin{displaymath}
    \left({\boldsymbol{\Delta^k \tilde{f}}}\right)\big|_{\mathcal{M}}\in \mathcal{E}(w-k, w'-k)
\end{displaymath}
depends only on ${\boldsymbol{{f}}}={\boldsymbol{\tilde{f}}}\big|_{\mathcal{M}}$ and defines a differential operator 
\begin{displaymath}
    {\boldsymbol{P}}_{w,w'}: \mathcal{E}(w, w') \to \mathcal{E}(w-k, w'-k)
\end{displaymath}
by ${\boldsymbol{P}}_{w,w'}=\frac{{\boldsymbol{\Delta^k}}}{2^k}$.
Moreover, the operator ${\boldsymbol{P}}_{w,w'}$ is independent of the choice of a Fefferman defining function if $k\leqslant n+1$.
\end{lemma}

It is known \cite[Proposition 5.1]{GG05} that ${\boldsymbol{P}}_{w,w'}$ defined in Lemma \ref{construction} is formally self-adjoint. In particular, when $w=w'=\frac{k-1-n}{2}$, $1\leqslant k\leqslant n+1$, $P_{k}^{\theta}=P_{w,w}: C^{\infty}(N)\to C^{\infty}(N)$ satisfies
\begin{equation}
    P_{k}^{e^{2f}\theta}u=e^{-(n+1+k)f}P_k^{\theta}\left( e^{(n+1-k)f}u\right), \quad f\in C^{\infty}(N,\R).
\end{equation}
For $k=1$, the operator $P_1^{\theta}$ agrees with the conformal sub-Laplacian of Jerison-Lee; i.e.
\begin{equation}\label{subla}
    P_1^{\theta}=\Delta_b+\frac{n}{4(n+1)}R
\end{equation}
where $R$ is the associated scalar curvature. In general, if we let $T$ denote the Reeb vector field of $\theta$, then $P_k^{\theta}$ has the same principal part (in the Heisenberg sense) as 
\begin{equation}\label{prin}
    \left(\Delta_b+i(k-1)T\right)\left(\Delta_b+i(k-3)T\right)\cdots\left(\Delta_b-i(k-1)T\right).
\end{equation}
This implies that the operator $P_k^{\theta}$ is subelliptic of order $2k$ when $k<n+1$, see \cite[Section 2.2.4]{Ponge05} for more details. We denote the principal part of $P_k^{\theta}$ by $\tilde{P}_k^{\theta}$.

\subsection{Factorization of CR GJMS operators on $S^{2n+1}$}
This subsection deals with the factorization of CR GJMS operators on $S^{2n+1}$; we follow the argument in \cite{Takeuchi17}.

Let $S^{2n+1}\subset \C^{n+1}$ be the unit sphere centered at the origin with canonical CR structure, and $\theta_0$ be the contact form on $S^{2n+1}$ defined by
\begin{equation}
    \theta_0=\frac{i}{2}\sum_{i=1}^{n+1}\left(z^i d\bar{z}^i-\bar{z}^i dz^i\right)\big|_{S^{2n+1}}.
\end{equation}
Then the triple $(S^{2n+1}, T^{1,0}S^{2n+1}, \theta_0)$ is a Sasakian--Einstein manifold; i.e. the Tanaka--Webster Ricci curvature ${\mathrm{Ric}}_{\alpha \bar{\beta}}$ of $\theta_0$ satisfies
\begin{displaymath}
    {\mathrm{Ric}}_{\alpha \bar{\beta}}=(n+1)h_{\alpha\bar{\beta}}
\end{displaymath}
where $(h_{\alpha\bar{\beta}})$ is a positive Hermitian matrix defined in (\ref{metric}). Consider the Tanaka--Webster connection with respect to $\theta_0$. Note that the index $0$ is used for the component $T$ or $\theta_0$ in our index notation. The commutators of the derivatives for $f\in C^{\infty}(S^{2n+1})$ are given by
\begin{displaymath}
f_{[\alpha\beta]}=0, \quad 2f_{[\alpha\bar{\beta}]}=i h_{\alpha\bar{\beta}}f_0, \quad f_{[0\alpha]}=0,
\end{displaymath}
where $[\cdots]$ means the antisymmetrization over the enclosed indices. The Kohn Laplacian $\Box_b$ and sub-Laplacian $\Delta_b$ are given by
\begin{equation}\label{B}
\Box_b f=-f_{\bar{\alpha}}^{\;\;\bar{\alpha}}, \quad \Delta_bf=-f_{\bar{\alpha}}^{\;\;\bar{\alpha}}-f_{{\alpha}}^{\;\;{\alpha}}=\Box_b f+\overline{\Box}_b f,
\end{equation}
respectively. From the above commutation relations, we have
\begin{equation}\label{T}
\Box_b-\overline{\Box}_b=inT.
\end{equation}
The third covariant derivatives of $f$ satisfy
\begin{displaymath}
\begin{split}
f_{\alpha[0\beta]}&=0, \quad f_{\alpha[0\bar{\beta}]}=0,\\
2f_{\alpha[\beta\bar{\gamma}]}&=ih_{\beta\bar{\gamma}}f_{\alpha 0}+R_{\alpha\;\beta \bar{\gamma}}^{\;\;\delta\;\;\;}f_{\delta}.
\end{split}
\end{displaymath}
From these, it follows that the Kohn Laplacian and sub-Laplacian commute with the Reeb vector field $T$.

The Fefferman defining function ${\boldsymbol{\rho}}$ of the null space $\mathcal{N}\subset \C^{n+2}$ is given by ${\boldsymbol{\rho}}(\zeta)= \zeta_0\bar{\zeta}_0-\sum_{j=1}^{n+1}\zeta_j\bar{\zeta}_j$,
where $\zeta$ is the homogeneous coordinate system defined by
\begin{equation}\label{homo}
    \zeta_0=z_0, \quad \zeta_j=z_0z_j, \quad j=1, \cdots, n+1.
\end{equation}
The corresponding ambient metric ${\boldsymbol{{g}[\rho]}}$ is a flat Lorentz-K\"ahler metric defined on $\C^{n+2}\backslash\{0\}$ by
\begin{equation}
    {\boldsymbol{{g}[\rho]}}=i\left(-d\zeta_0\wedge d\bar{\zeta}_0+d\zeta_1\wedge d\bar{\zeta}_1+\cdots d\zeta_{n+1}\wedge d\bar{\zeta}_{n+1}\right).
\end{equation}
In homogeneous coordinates, the holomorphic Euler field $Z$ and anti-holomorphic Euler field $\overline{Z}$ are given by
\begin{equation}
Z=\sum_{i=0}^{n+1}\zeta_i\frac{\partial}{\partial \zeta_i}, \quad \overline{Z}=\sum_{i=0}^{n+1}\bar{\zeta}_i\frac{\partial}{\partial \bar{\zeta}_i}.
\end{equation}
Notice that for each function ${\boldsymbol{f}}\in \tilde{\mathcal{E}}(w, w')$, there exists a smooth function $f$ defined on $\C^{n+1}\backslash\{0\}$ such that ${\boldsymbol{f}}=\left(z_0\right)^w \left(\bar{z}_0\right)^{w'}f$. For convenience, we define the multiplication operator ${\boldsymbol{M}}_{v,v'}$ to raise the lower the bidegree.

\begin{definition}
On the Lorentz-K\"ahler manifold $(\C^{n+2}\backslash\{0\}, {\boldsymbol{{g}[\rho]}})$, the operator ${\boldsymbol{M}}_{v,v'}:\tilde{\mathcal{E}}(w, w')\to \tilde{\mathcal{E}}(w+v, w'+v')$ is defined by the multiplication by $\left(z_0\right)^v\left(\bar{z}_0\right)^{v'}$.
\end{definition}
These multiplication operators ${\boldsymbol{M}}_{v,v'}$ induce differential operators on $\S^{2n+1}$ corresponding to ${\boldsymbol{P}}_{w,w'}$ defined in Lemma \ref{construction}.
\begin{definition}\label{P}
Let $(w,w')\in \R^2$ such that $k=w+w'+n+1$ is a positive integer. The differential operator $P_{w,w'}$ on $C^{\infty}(S^{2n+1})$ is defined by
\begin{equation}
    P_{w,w'}={\boldsymbol{M}}_{k-w,k-w'}\circ {\boldsymbol{P}}_{w,w'} \circ {\boldsymbol{M}}_{w,w'}.
\end{equation}
\end{definition}
The main result in \cite{Takeuchi17} is the factorization of $P_{w,w'}$ on a Sasakian $\eta$-Einstein manifold. In this paper, we only need the corresponding result on $S^{2n+1}$.
\begin{theorem}[{\cite{Takeuchi17}}]\label{fact}
Let $(S^{2n+1}, T^{1,0}S^{2n+1}, \theta_0)$ be a Sasakian manifold of dimension $2n+1$ with Einstein constant $(n+1)$. Then $P_{w,w'}$ for $k=w+w'+n+1$ has the formula
\begin{equation}\label{GJMS}
    P_{w,w'}=\prod_{j=0}^{k-1}L_{w'-w+k-2j-1}.
\end{equation}
Here $L_{\mu}$ is the differential operator on $\S^{2n+1}$ defined by
\begin{equation}\label{L}
    L_{\mu}=\frac{1}{2}\Delta_b+\frac{i}{2}\mu T+\frac{1}{4}(n-\mu)(n+\mu),
\end{equation}
where $\Delta_b$ is the sub-Laplacian operator and $T$ is the Reeb vector field.
\end{theorem}

\section{Improved CR Sobolev inequalities}
In this section, we improve the CR Sobolev inequalities on $S^{2n+1}$ under the vanishing of higher order moments of the volume elements. We start from the concentration compactness principle on CR manifolds $(N,\theta)$. The original version in \cite{Lions84,Lions85} is on $\R^n$, but it can be carried out on $(N,\theta)$ because, as mentioned in Remark I.5 in \cite{Lions84}, the crucial ingredient in Lemma I.2 there is valid in an arbitrary measure space.

\begin{lemma}\label{ccp}
Let $N$ be a compact strictly pseudoconvex CR manifold of dimension $2n+1$. We assume that $F_i$ is a sequence of complex-valued functions bounded in $S^{k,2}(N)$, $k\in\N$ with $k<n+1$ and $p=\frac{2Q}{Q-2k}$, such that $F_i\rightharpoonup F$ weakly in $S^{k,2}(N)$. Moreover, we assume as measures, 
\begin{equation}
|F_i|^p d\zeta \to |F|^pd\zeta+\nu, \quad \bar{F_i}\tilde{P}^{\theta}_k F_i d\zeta \to \bar{F}\tilde{P}^{\theta}_k F d\zeta+\sigma,
\end{equation}
where $d\zeta=\theta\wedge d\theta^n$. Then we can find countably many points $\{x_i\}\subset N$ such that
\begin{displaymath}
    \nu=\sum_{i}\nu_i \delta_{x_i}, \quad \nu_i^{\frac{2}{p}}\leqslant C_{n,2k} \sigma_i
\end{displaymath}
where $\nu_i=\nu(x_i)$ and $\sigma_i=\sigma(x_i)$.
\end{lemma}
\begin{proof}
Let $G_i:=F_i-F$. Then, up to a subsequence, as $i\to \infty$, $G_i\rightharpoonup 0$ weakly in $S^{k,2}(N)$ and $G_i\to 0$ in $L^2(N)$. Also, $\nu_i:=|F_i|^p d\zeta- |F|^pd\zeta \to \nu$ and $\sigma_i:=\bar{F_i}\tilde{P}^{\theta}_k F_i d\zeta - \bar{F}\tilde{P}^{\theta}_k F d\zeta\to \sigma$ as measures. For any measurable set $A\subset N$, by \cite[Lemma 2.6]{Lieb83} and the weak convergence in $S^{k,2}(N)$, respectively, we have 
\begin{equation}\label{mea}
\begin{split}
    \lim_{i\to \infty} \int_A |G_i|^p d\zeta&=\lim_{i\to \infty}\int_A|F_i|^p d\zeta- \int_A|F|^pd\zeta=\nu(A),\\
    \lim_{i\to \infty} \int_A \bar{G_i}\tilde{P}^{\theta}_k G_i d\zeta &= \lim_{i\to \infty} \int_A \bar{F_i}\tilde{P}^{\theta}_k F_i d\zeta - \int_A\bar{F}\tilde{P}^{\theta}_k F d\zeta=\sigma(A),
    \end{split}
\end{equation}
which implies that $|G_i|^p d\zeta \to \nu$ and $\bar{G_i}\tilde{P}^{\theta}_k G_i d\zeta\to \sigma$ as measures.

Firstly, we claim that $\sigma$ is a real measure. From \cite[Theorem 5.1]{GG05}, we know that the operator $P^{\theta}_k$ is self-adjoint. We have for any real function $\phi\in C^{\infty}(N)$,
\begin{equation}\label{self1}
\int_N\phi\bar{G_i}P^{\theta}_kG_i d\zeta=\int_N G_i P^{\theta}_k(\phi \bar{G_i}) d\zeta\\=\int_N \phi G_i P^{\theta}_k\bar{G_i} d\zeta+l.o.t,
\end{equation}
Moreover,
\begin{equation}\label{self2}
\begin{split}
    &\int_N\phi\bar{G_i}P^{\theta}_kG_i d\zeta= \int_N \phi \bar{G_i}\tilde{P}^{\theta}_k G_id\zeta+l.o.t,\\
    &\int_N \phi G_i P^{\theta}_k\bar{G_i} d\zeta=\int_N \phi G_i \tilde{P}^{\theta}_k\bar{G_i} d\zeta+l.o.t,
    \end{split}
\end{equation}
where all lower order terms in (\ref{self1}) and (\ref{self2}) involve lower order derivatives of $G_i$.

By the compact embedding $S^{k,2}(N)\hookrightarrow S^{k',2}(N)$ for $k'<k$ in \cite{FS74} and the weak convergence, all lower order terms in (\ref{self1}) and (\ref{self2}) will vanish as $i\to \infty$. Therefore, since $\bar{G_i}\tilde{P}^{\theta}_k G_i\theta\wedge d\theta^n\to \sigma$, combining (\ref{self1}) and (\ref{self2}), we have
\begin{equation}
    \int_N \phi d\overline{\sigma}=\int_N \phi d\sigma,
\end{equation}
which implies that $\sigma$ is a real measure.

Similarly, by the compact embedding $S^{k,2}(N)\hookrightarrow S^{k',2}(N)$ for $k'<k$ in \cite{FS74} and (\ref{mea}), for any real function $\varphi\in C^{\infty}(N)$, we have
\begin{equation}\label{convergence}
    \begin{split}
    & \lim_{i\to \infty}\int_N |G_i\varphi|^p \theta\wedge d\theta^n=\lim_{i\to \infty}\int_N|\varphi|^p d\nu_i=\int_N |\varphi|^p d\nu,\\
    & \lim_{i\to \infty}\int_N (\bar{G}_i\varphi) P^{\theta}_k(G_i\varphi) \theta\wedge d\theta^n=\lim_{i\to \infty}\int_N \varphi^2 \bar{G}_i\tilde{P}^{\theta}_k G_i \theta\wedge d\theta^n=\int_N \varphi^2 d\sigma.
    \end{split}
\end{equation}
Therefore, by the embedding $S^{k,2}(N)\hookrightarrow L^{p}(N)$ in \cite{FS74}, we obtain a reversed H\"older inequality; i.e. there exists a constant $C$ such that for any real function $\varphi\in C^{\infty}(N)$, 
\begin{displaymath}
    \left(\int_N |\varphi|^p d\nu\right)^{\frac{1}{p}}\leqslant C\left(\int_N \varphi^2 d\sigma\right)^{\frac{1}{2}}.
\end{displaymath}
Following the same argument in \cite[Lemma I.2]{Lions85}, we know that $\nu$ is supported on countably many points $\{x_i\}\subset N$. 

We claim that for each atom $x_j$, there does not exist a neighborhood of $x_j$ in which $\bar{G}_i\tilde{P}^{\theta}_k G_i$ is uniformly bounded. If this is not true, we assume that in a neighborhood $U$ of $x_j$, $\bar{G}_i\tilde{P}^{\theta}_k G_i$ has a uniform upper bound $C$. By Lebesgue's dominated convergence theorem, for any ball $B_R\subset U$ containing $x_j$, we have
\begin{equation}\label{unif}
    \sigma(B_R)=\lim_{i\to \infty}\int_{B_R} \bar{G}_i\tilde{P}^{\theta}_k G_i d\zeta=\int_{B_R} \lim_{i\to \infty} \bar{G}_i\tilde{P}^{\theta}_k G_i d\zeta\leqslant C\int_{B_R}d\zeta.
\end{equation}
By (\ref{unif}), we obtain that $\sigma(x_j)=0$ as $R\to 0$, which implies that $\sigma(x_j)$ and $\nu(x_j)$ are zero, a contradiction.

Therefore, there exists a sequence of points $\{\xi^{j}_i\}$ on $N$ such that as $i\to \infty$, 
\begin{equation}\label{blowup}
\xi^{j}_i\to x_j, \quad \left(\bar{G}_i\tilde{P}^{\theta}_k G_i\right)(\xi^j_i)\to \infty.
\end{equation} 
If $(N,\theta)$ is the unit sphere $S^{2n+1}$ with the canonical contact form $\theta_0$, without loss of generality, we can assume that the south pole $(0,\cdots, -1)$ is not an atom. Then, $(\H^n,\theta_c)$ and $(S^{2n+1},\theta_0)$ are related by the Cayley transform $\mathcal{C}:\H^n\to S^{2n+1}\backslash (0,\cdots, -1)$. Composed with a translation if necessary, $\mathcal{C}$ is a diffeomorphism of a neighborhood $\Omega_{\xi^j_i}$ of $\xi^j_i$ onto a neighborhood $U$ of the origin in $\H^n$. Moreover, it is known \cite{BFM07} that the contact form and CR GJMS operators on $\H^n$ and $S^{2n+1}$ are related by
\begin{align*}
    \mathcal{C}_{*} \theta_0&=\frac{2}{(1+|z|^2)^2+t^2}\theta_c,\\
    P^{\theta_c}_k\left(\left(2|J_{\mathcal{C}}|\right)^{\frac{Q-2k}{2Q}}\left(F\circ \mathcal{C}\right)\right)&=\left(2|J_{\mathcal{C}}|\right)^{\frac{Q+2k}{2Q}}\left(P^{\theta_0}F\right)\circ \mathcal{C},
\end{align*}
where $|J_{\mathcal{C}}|$ is the determinant of the Jacobian of $\mathcal{C}$ and $F\in C^{\infty}(S^{2n+1})$.

On general CR manifolds $(N,\theta)$, at each point $\xi^j_i$, we introduce the pseudohermitian normal coordinates \cite[Theorem 4.3]{JL87}. $\Theta_{\xi^j_i}$ is a diffeomorphism of a neighborhood $\Omega_{\xi^j_i}$ of $\xi^j_i$ onto a neighborhood $U$ of the origin in $\H^n$.

 We will use $\mathcal{C}$ or $\Theta_{\xi^j_i}$ to identify $U$ with the neighborhood $\Omega_{\xi^j_i}$ of ${\xi^j_i}$. Therefore, under the dilation on $\H^n$, $T^{\delta}(z,t)=(\delta^{-1}z, \delta^{-2}t)$, for $\delta$ sufficiently small, we have
\begin{equation}\label{expansion}
\begin{split}
        &(T^{\delta}\circ \Theta_{\xi^j_i})_{*}\theta=\delta^2(1+\delta O^1)\theta_c,\\
        &(T^{\delta}\circ \Theta_{\xi^j_i})_{*}(\theta\wedge d\theta^n)=\delta^{2n+2}(1+\delta O^1)\theta_c \wedge d\theta_c^n,\\
        &(T^{\delta}\circ \Theta_{\xi^j_i})_{*}\Delta_b=\delta^{-2}(P^{\theta_c}_2+\mathfrak{C}(\partial_z)+\delta O^1 \mathfrak{C}(\partial_t,\partial^2_z)+\delta^2 O^2\mathfrak{C}(\partial_z \partial_t)+\delta^3O^3\mathfrak{C}(\partial^2_t)),
        \end{split}
\end{equation}
where $P^{\theta_c}_2$ is the sub-Laplacian on $\H^n$ and $O^r\mathfrak{C}$ indicates an operator involving linear combinations of the indicated derivatives with coefficients in $O^r$. For convenience, we consider general CR manifolds because all arguments in the following hold on $(S^{2n+1},\theta_0)$ by replacing $\Theta_{\xi^j_i}$ by $\mathcal{C}$.

Now consider the change of coordinates $(\tilde{z},\tilde{t})=T^{\delta}(z,t)=(\delta^{-1}z,\delta^{-2}z)$ on $\H^n$, and set
\begin{displaymath}
    \tilde{\theta}_c=d\tilde{t}+i\sum^{n}_j\left(\tilde{z}^jd\bar{\tilde{z}}^j-\bar{\tilde{z}}^jd\tilde{z}^j\right)=\delta^{-2}T^{\delta}_{*}\theta_c, \quad T^{\delta}_{*}=\left(\left(T^{\delta}\right)^{-1}\right)^{*}.
\end{displaymath}
On the set $\delta^{-1}U$ with coordinates $(\tilde{z},\tilde{t})$ define $\tilde{G}_i(\tilde{z},\tilde{t})=\delta^{\frac{Q}{p}}G_i(\delta\tilde{z},\delta^2\tilde{t})$ with $\delta=\delta_i>0$ chosen so that $\left(\bar{\tilde{G}}_i\tilde{P}^{\tilde{\theta}_c}_k \tilde{G}_i\right)(0)=1$. Observe that $\theta=\theta_c$ at $0$, from (\ref{expansion}), we have
\begin{displaymath}
    \left(\bar{\tilde{G}}_i\tilde{P}^{\tilde{\theta}_c}_k \tilde{G}_i\right)(0)=\delta_i^{Q}\left(\bar{G}_i\tilde{P}^{\theta}_k G_i\right)(\xi^j_i).
\end{displaymath}
In particular, $\delta_i\to 0$ as $i\to \infty$, and hence $\delta_i^{-1}U$ tends to the full space $\H^n$ as $i\to \infty$. We define $\theta_i=\delta_i^{-2}(T^{\delta_i}\circ \Theta_{\xi^j_i})_{*}\theta$ in coordinates $(\tilde{z},\tilde{t})$ on the region $\delta_i^{-1}U$. 

If we denote by $\left(W^i_1, \cdots, W^i_n \right)$ the pseudohermitian frame used to define $\Theta_{\xi^j_i}$, we may assume that $\left(W^i_1, \cdots, W^i_n\right)$ converges in $C^r$ for all $r$ to a frame $\left(W_1, \cdots, W_n\right)$. Now set $Z^i_l=\delta_i (T^{\delta_i}\circ \Theta_{\xi^j_i})_{*} W^i_l$, so that $\left(Z^i_1, \cdots, Z^i_n\right)$ is a pseudohermitian frame for $\theta_i$.  By examining the error terms in the expression for $W^i_l$ \cite[Remark 4.4]{JL87}, it is easy to show that, for any $R>0$, $Z^i_l$ converges in $C^r(B_R)$ to $Z_l$ for every $r$. Similarly, $\theta_i$ and $P^{\theta_i}_k$ converge uniformly in $C^r(B_R)$ to $\theta_c$ and $P^{\theta_c}_k$, respectively.

Now fix a radius $R>0$, note that
\begin{equation}\label{trans}
    \int_{B_R(0)}|\tilde{G}_i(\tilde{z},\tilde{t})|^p \theta_i\wedge d\theta_i^n=\int_{B_{\delta_i R}(\xi^j_i)}|G_i(z,t)|^p \theta\wedge d\theta^n.
\end{equation}
The volume element $\theta\wedge d\theta^n$ is equal to $C\left(1+\delta_i O^1\right)\theta_c\wedge d\theta_c^n$ on $B_{2\delta_iR}$ by (\ref{expansion}). Therefore, $\tilde{G}_i\in L^p(B_{2R},\theta_i)$ with uniform bounds on the norm. In particular, $\tilde{G}_i\in L^k(B_{2R},\theta_i)$ uniformly as $i\to \infty$. Combined with the uniform bound on $\bar{\tilde{G}}_i\tilde{P}^{\theta_i}_k \tilde{G}_i$, this gives $\tilde{G}_i\in S^{k,r}(B_{2R}, \theta_i)$ for every $r<\infty$ with uniform bounds on the norm by \cite[Section 19]{FS74}. Thus by \cite[Proposition 5.15]{JL87}, $\tilde{G}_i$ is uniformly bounded in $C^r(B_R)$ for every $r$.

Now we can take a subsequence $i_l\to \infty$ for which $\tilde{G}_{i_l}$ converges in $C^1(B_R)$. Define a function $G$ on all of $\H^n$ by first choosing a subsequence $\tilde{G}_{i_l}$ converging in $C^1(B_1)$, and then a subsequence converging in $C^1(B_2)$, etc. 

Notice that $G\in C^1(\H^n)$ and $G\neq 0$ because $\left(\bar{G}\tilde{P}^{\theta_c}_kG\right)(0)=1$. Since $\theta_i\wedge d\theta_i^n$ approaches $\theta_c\wedge d\theta_c$ uniformly on compact sets, we verify
\begin{displaymath}
    \int_{\H^n} \bar{G}P^{\theta_c}_kG \theta_c\wedge d\theta_c\leqslant C<\infty.
\end{displaymath}
In fact, 
\begin{equation}\label{trans2}
\begin{split}
    \int_{B_R} \bar{G}P^{\theta_c}_kG \theta_c\wedge d\theta_c&=\lim_{i\to \infty}\int_{B_R}\bar{\tilde{G}}_i\tilde{P}^{\theta_i}_k \tilde{G}_i \theta_i\wedge d\theta_i^n\\
    &=\lim_{i\to \infty}\int_{B_{\delta_i R}} \delta_i^{Q}\left(\bar{G}_i\tilde{P}^{\theta}_k G_i\right) (\delta_i^{-2}\theta)\wedge \left(\delta_i^{-2}d\theta\right)^n+ l.o.t\\
    &\leqslant \overline{\lim_{i\to \infty}} \int_{N} \bar{G}_i\tilde{P}^{\theta}_k G_i \theta\wedge d\theta^n<\infty
    \end{split}
\end{equation}
because all lower order terms will vanish as $\delta_i\to 0$ by (\ref{expansion}). Since $\tilde{G}_i$ converges to $G$ pointwisely on $\H^n$ and $G\in S^{k,2}(\H^n)$, by Lebesgue's dominated convergence theorem, for fixed radius $R$ such that $B_R\subset \Omega_{\xi^j_i}$ for all $i$, we have
\begin{equation}\label{e1}
\begin{split}
    & \lim_{i\to \infty}\int_{\delta_i^{-1}B_R}|\tilde{G}_i|^p \theta_i\wedge d\theta_i^n=\int_{\H^n}|G|^p \theta_c\wedge d\theta_c,\\
    & \lim_{i\to \infty}\int_{\delta_i^{-1}B_R}\bar{\tilde{G}}_i\tilde{P}^{\theta_i}_k \tilde{G}_i \theta_i\wedge d\theta_i^n= \int_{\H^n} \bar{G}P^{\theta_c}_kG \theta_c\wedge d\theta_c.
    \end{split}
\end{equation}
Moreover, by (\ref{optimal}), 
\begin{equation}\label{e2}
    \left(\int_{\H^n}|G|^p \theta_c\wedge d\theta_c\right)^{\frac{2}{p}}\leqslant C_{n,2k}\int_{\H^n} \bar{G}P^{\theta_c}_kG \theta_c\wedge d\theta_c.
\end{equation}
Besides, by similar arguments in (\ref{trans}) and (\ref{trans2}), we obtain
\begin{equation}\label{e3}
    \begin{split}
        & \lim_{i\to \infty}\int_{\delta_i^{-1}B_R}|\tilde{G}_i|^p \theta_i\wedge d\theta_i^n=\lim_{i\to \infty}\int_{B_R}|G_i|^p \theta\wedge d\theta^n=\nu (B_R),\\
        & \lim_{i\to \infty}\int_{\delta_i^{-1}B_R} \bar{G}P^{\theta_c}_kG \theta_c\wedge d\theta_c=\lim_{i\to \infty}\int_{B_{R}} \bar{G}_i\tilde{P}^{\theta}_k G_i \theta\wedge d\theta^n=\sigma (B_R).
    \end{split}
\end{equation}
Combining (\ref{e1}), (\ref{e2}) and (\ref{e3}), as $R\to 0$ in (\ref{e3}), we obtain
\begin{equation}\label{e4}
    (\nu (x_j))^{\frac{2}{p}}\leqslant C_{n,2k}\sigma (x_j).
\end{equation}
Applying the above argument at each point $x_j\in N$, we complete the proof.
\end{proof}

Now we can prove Theorem \ref{improve} with the help of Lemma \ref{ccp}. 
\begin{proof}[Proof of Theorem \ref{improve}]
Let 
\begin{displaymath}
    \alpha=\frac{{C}_{n,2k}}{\Theta({w,w'};\frac{Q-2k}{Q},2n+1)}+\epsilon.
\end{displaymath}
If (\ref{asi}) is not true, then for any $j\in \N$, we can find a $F_j\in S^{k,2}(S^{2n+1})$ such that 
\begin{equation}\label{constraint}
    \int_{S^{2n+1}}g|F|^p d\xi=0
\end{equation}
for all $g\in \overline{\PP}_{w,w'}$, and 
\begin{equation}
    \left(\int_{S^{2n+1}}|F_j|^p d\xi\right)^{\frac{2}{p}}> \alpha \int_{S^{2n+1}}\bar{F}_j P_k^{\theta_0} F_j d\xi+j\int_{S^{2n+1}}|F_j|^2 d\xi.
\end{equation}
We may assume
\begin{displaymath}
    \left(\int_{S^{2n+1}}|F_j|^p d\xi\right)^{\frac{2}{p}}=1.
\end{displaymath}
Then
\begin{equation}\label{small}
    \int_{S^{2n+1}}\bar{F}_j P_k^{\theta_0} F_j d\xi<\frac{1}{\alpha}, \quad \int_{S^{2n+1}}|F_j|^2 d\xi<\frac{1}{j}.
\end{equation}
It follows that $F_j\rightharpoonup 0$ weakly in $S^{k,2}(S^{2n+1})$. After passing to a subsequence we have
\begin{equation}
    \bar{F}_j P_k^{\theta_0} F_j d\xi\to \sigma, \quad |F_j|^p d\xi \to \nu.
\end{equation}
By Lemma \ref{ccp} we can find countably many points $\{x_i\}\in S^{2n+1}$ such that
\begin{equation}\label{bubble}
    \nu=\sum_{i}\nu_i \delta_{x_i}, \quad \nu_i^{\frac{2}{p}}\leqslant {C}_{n,2k} \sigma_i,
\end{equation}
where $\nu_i=\nu(x_i)$, $\sigma_i=\sigma(x_i)$. Then
\begin{equation}\label{volume}
    \nu(S^{2n+1})=1,\quad \sigma(S^{2n+1})<\frac{1}{\alpha}.
\end{equation}
It follows from (\ref{constraint}) and (\ref{small}) that $\int_{S^{2n+1}}gd\nu=0$ for all $g\in \overline{\PP}_{w,w'}$, hence $\nu\in \mathcal{M}^c_{w,w'}(S^{2n+1})$. By definition of $\Theta(w,w';\theta,2n+1)$, (\ref{bubble}) and (\ref{volume}) we have
\begin{displaymath}
    \Theta({w,w'};\frac{Q-2k}{Q},2n+1)\leqslant \sum_i \nu_i^{\frac{2}{p}}\leqslant \sum_i {C}_{n,2k} \sigma_i={C}_{n,2k} \sigma(S^{2n+1})\leqslant \frac{{C}_{n,2k}}{\alpha}.
\end{displaymath}
Hence
\begin{displaymath}
    \alpha\leqslant \frac{{C}_{n,2k}}{\Theta({w,w'};\frac{Q-2k}{Q},2n+1)}.
\end{displaymath}
This contradicts the choice of $\alpha$.
\end{proof}
We now compute $\Theta (1,1; \theta, 2n+1)$ following the idea in \cite{HW21}.
\begin{proposition}\label{value}
For $\theta\in (0,1)$ and $n\in \N$,
\begin{equation}
    \Theta (1,1; \theta, 2n+1)=(n+2)^{1-\theta}.
\end{equation}
Moreover $\Theta (1,1; \theta, 2n+1)$ is achieved at $v\in \mathcal{M}^c_{1,1}(S^{2n+1})$ if and only if
\begin{equation}\label{achi}
    \nu=\frac{1}{n+2}\sum_{i=1}^{n+2}\delta_{x_i}
\end{equation}
for $\{x_i\}\in S^{2n+1}$ being the vertices of a regular $(n+1)$-simplex embedded in the unit ball.
\end{proposition}
Before the proof, we make some observations about $\mathcal{M}^c_{1,1}(S^{2n+1})$. As complex vector spaces, 
\begin{equation}\label{base}
    \begin{split}
        \PP_{1,0}&=\mathrm{span} \left\{ z_1, \cdots, z_{n+1},\right\},\quad \PP_{0,1}=\mathrm{span} \left\{ \bar{z}_1, \cdots, \bar{z}_{n+1},\right\},\\
        \PP_{1,1}&=\mathrm{span} \left\{z_1\bar{z}_1,\cdots, z_n\bar{z}_n, z_i\bar{z}_j, \quad 1\leqslant i,j\leqslant n+1, i\neq j\right\}.
    \end{split}
\end{equation}
Let $\nu$ be a probability measure supported on countably many points $\{x_i\}_{i=1}^{K}\subset S^{2n+1}$. We denote $\nu_i=\nu(x_i)$, and take $\nu(x_i)=0$ for $i>K$ if $K$ is finite. Define $n+2$ vectors in $\ell^2$ as
\begin{align}\label{vectors}
\begin{split}
    &v_0=\begin{pmatrix}
    \sqrt{\nu_1}, \sqrt{\nu_2}, \cdots
    \end{pmatrix},\\
    &v_j=\begin{pmatrix}
    \sqrt{(n+1)\nu_1}x_{1,j}, \sqrt{(n+1)\nu_2}x_{2,j}, \cdots
    \end{pmatrix},\quad 1\leqslant j\leqslant n+1.\\
    \end{split}
\end{align}
Here $x_{i,j}$ is the $j$th coordinate of $x_i$ as a vector in $\C^{n+1}$ and
\begin{displaymath}
    \ell^2=\left\{(c_1,c_2,\cdots):c_i\in \C, \sum_{i=1}^{\infty}c_i\bar{c}_i<\infty\right\}
\end{displaymath}
is equipped with the stand inner product.
\begin{lemma}\label{ortho}
Let $\nu$ be a probability measure supported on countably many points $\{x_i\}_{i=1}^{K}\subset S^{2n+1}$. Then
\begin{center}
    $\nu\in\mathcal{M}^c_{1,1}(S^{2n+1})\Rightarrow v_0,\cdots, v_{n+1}$ is orthonormal in $\ell^2$.
\end{center}
Here $v_0,\cdots, v_{n+1}$ is defined in (\ref{vectors}) and $\ell^2$ is equipped with standard inner product.
\end{lemma}
\begin{proof}
Notice that $\widetilde{\PP}_{1,1}=\PP_{1,0}\cup \PP_{0,1}\cup \PP_{1,1}$. By the definition of $\mathcal{M}^c_{1,1}(S^{2n+1})$ and (\ref{base}), $\nu$ satisfies
\begin{align}\label{orthor}
    \int_{S^{2n+1}} \left(z_i\bar{z}_i-\frac{1}{n+1}\right)d\nu=0, \quad \int_{S^{2n+1}} z_i\bar{z}_jd\nu=0, \quad 1\leqslant i,j\leqslant n+1, i\neq j.
\end{align}
The orthonormality of $\{v_j\}_{j=0}^{n+1}$ follows from (\ref{orthor}).
\end{proof}

\begin{proof}[Proof of Proposition \ref{value}]
If $\nu\in\mathcal{M}^c_{1,1}(S^{2n+1})$ is a probability measure supported on countably many points $\{x_i\}_{i=1}^{K}\subset S^{2n+1}$, then by Lemma \ref{ortho} we know that $v_0,\cdots, v_{n+1}$ is orthonormal in $\ell^2$. Let $e_1, e_2, \cdots$ be the standard base of $\ell^2$. Then for $1\leqslant i\leqslant K$, 
\begin{equation}
    \sum_{j=0}^{n+1}\langle e_i, v_j\rangle ^2=(n+2)\nu_i.
\end{equation}
Here $\langle \cdot, \cdot\rangle$ is the standard inner product on $\ell^2$. It follows from Parseval's relation that 
\begin{equation}
    0<\sum_{j=0}^{n+1}\langle e_i, v_j\rangle ^2\leqslant 1.
\end{equation}
Hence using $0<\theta<1$, we have
\begin{equation}
    \begin{split}
        \sum_{i=1}^K\nu_i^{\theta}&=\frac{1}{(n+2)^{\theta}}\sum_{i=1}^K\left(\sum_{j=0}^{n+1}\langle e_i, v_j\rangle ^2\right)^{\theta}\\
        & \geqslant \frac{1}{(n+2)^{\theta}}\sum_{i=1}^K\sum_{j=0}^{n+1}\langle e_i, v_j\rangle ^2\\
        & =\frac{1}{(n+2)^{\theta}}\sum_{j=0}^{n+1}\sum_{i=1}^K\langle e_i, v_j\rangle ^2\\
        & =\frac{1}{(n+2)^{\theta}}\sum_{j=0}^{n+1}1\\
        & =(n+2)^{1-\theta}.
    \end{split}
\end{equation}
If the equality holds, then for $1\leqslant i\leqslant K$,
\begin{displaymath}
    \sum_{j=0}^{n+1}\langle e_i, v_j\rangle ^2= 1.
\end{displaymath}
Therefore, $\nu_i=\frac{1}{n+2}$ and $K=n+2$. Moreover, it implies the matrix
\begin{displaymath}
    A=\begin{bmatrix}
    v_0, v_1, \cdots, v_{n+1}
    \end{bmatrix}
\end{displaymath}
is unitary. We know for $1\leqslant i<j\leqslant n+2$, $\left\|x_i-x_j\right\|=\sqrt{\frac{2(n+2)}{n+1}}$. Hence, $\nu=\frac{1}{n+2}\sum_{i=1}^{n+2}\delta_{x_i}$
for $\{x_i\}\in S^{2n+1}$ being the vertices of a regular $(n+1)$-simplex embedded in the unit ball.

Finally, we see $\Theta (1,1; \theta, 2n+1)\geqslant (n+2)^{1-\theta}$, and it can be achieved by a probability measure $\nu$ in the form of (\ref{achi}). It follows that $\Theta (1,1; \theta, 2n+1)=(n+2)^{1-\theta}.$
\end{proof}

\section{Sharp CR Sobolev inequalities}
The main result in this section is the classification of optimizers of the sharp Sobolev inequalities (\ref{sharpSo}). First, by improved Sobolev inequalities in (\ref{balanced}), we prove the existence of an optimizer.
\begin{proposition}
Denote $Q=2n+2$, $p=\frac{2Q}{Q-2k}$. Then the sharp constant in (\ref{sharpSo}) is attained. Moreover, for any minimizing sequence $\{F_i\}$ there is a subsequence $\{F_{i_m}\}$ and a sequence $\{\Phi_{i_m}\}$ in the CR automorphism group $\mathcal{A}(S^{2n+1})$ of $S^{2n+1}$ such that
\begin{equation}\label{replaced}
    F^{\Phi}_{i_m}=|J_{\Phi_{i_m}}|^{\frac{1}{p}}\Phi^{*}_{i_m}F_{i_m}
\end{equation}
converges strongly in $S^{k,2}(S^{2n+1})$, where $|J_{\Phi_{i_m}}|$ is the determinant of the Jacobian of $\Phi_{i_m}$.
\end{proposition}
\begin{proof}
By Lemma B.1 in \cite{FL10}, for each $F_i$, there exists an element $\Phi_{i}$ of $\mathcal{A}(S^{2n+1})$ such that $F^{\Phi}_{i}$ defined in (\ref{replaced}) satisfies the balanced conditions (\ref{balanced}). By CR invariance, we may replace $F_i$ by $F^{\Phi}_{i}$ in (\ref{sharpSo}). Thus, $\{F^{\Phi}_{i}\}$ is also a minimizing sequence. Passing to a subsequence $\{F_{i_m}\}$, we may assume that $F^{\Phi}_{i_m}\rightharpoonup F$ weakly in $S^{k,2}(S^{2n+1})$ and that $F^{\Phi}_{i_m} \to F$ strongly in $L^2(S^{2n+1})$. 

Without loss of generality, we may assume $\int_{S^{2n+1}}|F^{\Phi}_{i_m}|^p d\xi=1$ for all $i_m$. By Corollary \ref{impro}, we can choose a sufficiently small $\epsilon$ such that as $i_m\to \infty$,
\begin{equation}\label{nonzero}
    \begin{split}
        1=\int_{S^{2n+1}}|F^{\Phi}_{i_m}|^p d\xi&\leqslant \left(\frac{{C}_{n,2k}}{2^{\frac{k}{n+1}}}+\epsilon\right)\int_{S^{2n+1}}\bar{F}^{\Phi}_{i_m} P_k^{\theta_0} F^{\Phi}_{i_m} d\xi +C(\epsilon)\int_{S^{2n+1}}|F^{\Phi}_{i_m}|^2 d\xi,\\
        &\leqslant \left({C}_{n,2k}-\epsilon\right)\int_{S^{2n+1}}\bar{F}^{\Phi}_{i_m} P_k^{\theta_0} F^{\Phi}_{i_m} d\xi+C(\epsilon)\int_{S^{2n+1}}|F^{\Phi}_{i_m}|^2 d\xi,\\
        &\to \left(1-\frac{\epsilon}{{C}_{n,2k}}\right)+C(\epsilon)\int_{S^{2n+1}}|F|^2 d\xi.
    \end{split}
\end{equation}
which implies $F\neq 0$.

By \cite[Lemma2.6]{Lieb83}, we have
\begin{equation}\label{dep}
        1=\|F^{\Phi}_{i_m}\|_p^p=\|F\|_p^p+\|F^{\Phi}_{i_m}-F\|_p^p+o(1),
\end{equation}
Since for $a,b,c\geqslant 0$ and $p> 2$, $\left(a^p+b^p+c^p\right)^{\frac{2}{p}}\leqslant a^2+b^2+c^2$. We have 
\begin{equation}\label{str}
    \begin{split}
        \left(\int_{S^{2n+1}}|F^{\Phi}_{i_m}|^p d\xi\right)^{\frac{2}{p}}&- \left(\int_{S^{2n+1}}|F|^p d\xi\right)^{\frac{2}{p}}\\
        &\leqslant \left(\int_{S^{2n+1}}|F^{\Phi}_{i_m}-F|^p d\xi\right)^{\frac{2}{p}}+o(1)\\
        &\leqslant {C}_{n,2k}\int_{S^{2n+1}}\left(\bar{F}^{\Phi}_{i_m}-\bar{F}\right) P_k^{\theta_0} \left(F^{\Phi}_{i_m}-F\right) d\xi +o(1)\\
        &\leqslant {C}_{n,2k}\left(\int_{S^{2n+1}}\bar{F}^{\Phi}_{i_m} P_k^{\theta_0} F^{\Phi}_{i_m} d\xi- \int_{S^{2n+1}}\bar{F} P_k^{\theta_0} F d\xi\right)+o(1).
    \end{split}
\end{equation}
Since $\{F^{\Phi}_{i_m}\}$ is a minimizing sequence, as $i_m\to \infty$, we conclude that
\begin{equation}
    1-\left(\int_{S^{2n+1}}|F|^p d\xi\right)^{\frac{2}{p}}\leqslant 1-{C}_{n,2k}\int_{S^{2n+1}}\bar{F} P_k^{\theta_0} F d\xi.
\end{equation}
This implies that $F$ is a minimizer because $F\neq 0$.

In order to see that the convergence of $\{F^{\Phi}_{i_m}\}$ in $L^p(S^{2n+1})$ is strong, we need to show that $\|F\|_p^p=1$. By the weak convergence and (\ref{dep}), we may assume that $\|F\|_p^p=a\in (0,1]$ and $\lim \|F^{\Phi}_{i_m}-F\|_p^p=1-a$. The fact that $F$ is a minimizer implies equalities in (\ref{str}) in limit; i.e.
\begin{displaymath}
    1-a^{\frac{2}{p}}=(1-a)^{\frac{2}{p}}=1-a^{\frac{2}{p}}.
\end{displaymath}
This gives the conclusion because $1<a^{\frac{2}{p}}+(1-a)^{\frac{2}{p}}$ for $a\in (0,1)$.
\end{proof}
\begin{corollary}\label{transl}
Let $(S^{2n+1}, T^{1,0}S^{2n+1}, \theta_0)$ be the CR unit sphere with the volume element $d\xi$. Suppose that $u$ is a positive local minimizer of \begin{displaymath}
Y_k(S^{2n+1})=\inf \left\{A^{\theta_0}_k(F);F\in S^{k,2}(S^{2n+1}), B^{\theta_0}_k(F)=1 \right\}, \quad 1\leqslant k<n+1,
\end{displaymath}
where $A^{\theta_0}_k(F)$ and $B^{\theta_0}_k(F)$ are defined by 
\begin{align*}
    A^{\theta_0}_k(F)& :=\int_{S^{2n+1}} \bar{F} P_k^{\theta_0}(F) d\xi, \quad F\in S^{k,2}(S^{2n+1}),\\
    B^{\theta_0}_k(F)& :=\int_{S^{2n+1}} |F|^p d\xi, \quad p=\frac{2Q}{Q-2k}.
\end{align*}
Then there is an element $\Phi$ of the CR automorphism group ${\mathcal{A}}(S^{2n+1})$ of $S^{2n+1}$ such that
\begin{equation}\label{replace}
    u^{\Phi}=|J_{\Phi}|^{\frac{1}{p}}\Phi^{*}u
\end{equation}
is a positive minimizer of $Y_k(S^{2n+1})$ which satisfies (\ref{balanced}), where $|J_{\Phi}|$ is the determinant of the Jacobian of $\Phi$.
\end{corollary}
Next, we classify optimizers of sharp Sobolev inequalities (\ref{sharpSo}) following Frank--Lieb argument developed in \cite{Case19}. As explained in the introduction, the proof of Theorem \ref{class} consists of three ingredients. The desired commutator identity and spectral estimate are given in the following two results.
\begin{theorem}\label{commutator}
Let $(S^{2n+1}, T^{1,0}S^{2n+1}, \theta_0)$ be the CR unit sphere and $(\C^{n+2}\backslash\{0\}, {\boldsymbol{g[\rho]}})$ be its ambient space. Let $ P_{w,w'}$ denote the GJMS operator given in (\ref{GJMS}). Then
\begin{equation}
    \sum_{j=1}^{n+1}\bar{z}_j \left[P_{w,w'},z_j\right]=-k(w'-k+1) P_{w-1,w'}.
\end{equation}
\end{theorem}
\begin{proposition}\label{spectral}
Let $(S^{2n+1}, T^{1,0}S^{2n+1}, \theta_0)$ be the CR unit sphere with the volume element $d\xi$. Suppose that $u$ is a positive local minimizer of $Y_{k}(S^{2n+1})$. Suppose additionally that 
\begin{equation}\label{cbalanced}
    \int_{S^{2n+1}}z_j u^p d\xi=0, \quad j=1,\cdots, n+1,
\end{equation}
where $z_1,\cdots, z_{n+1}$ are coordinates on $\C^{n+1}$ and we regard $S^{2n+1}\subset \C^{n+1}$ as the unit sphere. Then
\begin{equation}\label{spectral3}
    \sum_{j=1}^{n+1}\int_{S^{2n+1}}\bar{z}_j u[P_k^{\theta_0},z_j](u)d\xi \geqslant (p-2) \int_{S^{2n+1}}uP^{\theta_0}_k(u)d\xi,
\end{equation}
where 
\begin{displaymath}
    [P_k^{\theta_0},z_j](u):=P_k^{\theta_0}(z_j u)-z_jP_k^{\theta_0}(u).
\end{displaymath}
\end{proposition}
The proof of Theorem \ref{commutator} needs the following lemma. 
\begin{lemma}\label{com}
Let $(S^{2n+1}, T^{1,0}S^{2n+1}, \theta_0)$ be the CR unit sphere and $(\C^{n+2}\backslash\{0\}, {\boldsymbol{g[\rho]}})$ be its ambient space. Given integers $k\geqslant 1$ and $1\leqslant j\leqslant n+1$, it holds that
\begin{equation}
    \zeta_0 \boldsymbol{\Delta}^k\zeta_j-\zeta_j \boldsymbol{\Delta}^k\zeta_0=-2k\boldsymbol{\mathcal{L}}_{{X_j}}\circ \boldsymbol{\Delta}^{k-1},
\end{equation}
where $\zeta$ is the homogeneous coordinate system defined in (\ref{homo}), and the vector field $X_j$ is defined by
\begin{equation}
    X_j=\zeta_0\boldsymbol{\nabla}\zeta_j-\zeta_j\boldsymbol{\nabla}\zeta_0.
\end{equation}
\end{lemma}
\begin{proof}
We begin with a simple observation about commutators of the ambient Laplacian. Since ${\boldsymbol{g[\rho]}}$ is a flat metric in homogeneous coordinates $\zeta$, for coordinate functions $\zeta_i$, $0\leqslant i\leqslant n+1$, we have 
\begin{equation}
    \left[\boldsymbol{\Delta}, \zeta_i\right]=-2\boldsymbol{\mathcal{L}}_{\boldsymbol{\nabla}\zeta_i}, \quad \left[\boldsymbol{\nabla}, \boldsymbol{\mathcal{L}}_{\boldsymbol{\nabla}\zeta_i}\right]=0. 
\end{equation}
By induction, we have that 
\begin{equation}\label{com1}
    \left[\boldsymbol{\Delta}^k, \zeta_i\right]=-2k \boldsymbol{\mathcal{L}}_{\boldsymbol{\nabla}\zeta_i}\circ \boldsymbol{\Delta}^{k-1}, \quad k\in \N.
\end{equation}
Therefore, by (\ref{com1}), we have
\begin{displaymath}
    \begin{split}
        \zeta_0 \boldsymbol{\Delta}^k\zeta_j-\zeta_0 \zeta_j \boldsymbol{\Delta}^k&=\zeta_0 \left[\boldsymbol{\Delta}^k, \zeta_j\right]=-2k\boldsymbol{\mathcal{L}}_{\zeta_0\boldsymbol{\nabla}\zeta_j}\circ \boldsymbol{\Delta}^{k-1},\\
        \zeta_j\zeta_0 \boldsymbol{\Delta}^k-\zeta_j \boldsymbol{\Delta}^k\zeta_0&=-\zeta_j \left[\boldsymbol{\Delta}^k, \zeta_0\right]=2k\boldsymbol{\mathcal{L}}_{\zeta_j\boldsymbol{\nabla}\zeta_0}\circ \boldsymbol{\Delta}^{k-1}.
    \end{split}
\end{displaymath}
The final conclusion follows from the definition of $X_j$.
\end{proof}
We now prove the commutator identity involving the GJMS operators needed to execute Frank-Lieb argument.
\begin{proof}[Proof of Theorem \ref{commutator}]
Applying Lemma \ref{com} to each $\zeta_j$, $1\leqslant j\leqslant n+1$, multiplying $\bar{\zeta}_j$ and taking the sum, we obtain
\begin{equation}\label{com3}
    \sum_{j=1}^{n+1} \left(\zeta_0 \bar{\zeta}_j \boldsymbol{\Delta}^k\zeta_j-\zeta_j \bar{\zeta}_j\boldsymbol{\Delta}^k\zeta_0\right)=-2k\boldsymbol{L}_{\sum_{j=1}^{n+1}\bar{\zeta}_jX_j}\circ \boldsymbol{\Delta}^{k-1}.
\end{equation}
By the definition of the homogeneous coordinates (\ref{homo}), we notice that both sides of (\ref{com3}) map $\tilde{\mathcal{E}}(w-1, w')$ to $\tilde{\mathcal{E}}(w-k, w'-k+1)$. Specializing to ${\boldsymbol{P}}_{w,w'}$ yields
\begin{equation}
    \sum_{j=1}^{n+1} \zeta_0 \bar{\zeta}_j {\boldsymbol{P}}_{w,w'}\zeta_j=\zeta_0 \bar{\zeta}_0{\boldsymbol{P}}_{w,w'}\zeta_0-k{\boldsymbol{L}}_{\sum_{j=1}^{n+1}\bar{\zeta}_jX_j}\circ {\boldsymbol{P}}_{w-1,w'}.
\end{equation}
Since ${\boldsymbol{g[\rho]}}$ is a flat Lorentz-K\"ahler metric, we find that on $\mathcal{N}$
\begin{equation}
\begin{split}
   \sum_{j=1}^{n+1}\bar{\zeta}_jX_j&=\sum_{j=1}^{n+1} \left(\zeta_0\bar{\zeta}_j\boldsymbol{\nabla}\zeta_j-\zeta_j\bar{\zeta}_j\boldsymbol{\nabla}\zeta_0\right)\\
   &=\sum_{j=1}^{n+1} \left(\zeta_0\bar{\zeta}_j\frac{\partial}{\partial \bar{\zeta}_j}+\zeta_j\bar{\zeta}_j\frac{\partial}{\partial \bar{\zeta}_0}\right)\\
   &=\zeta_0\sum_{j=0}^{n+1} \bar{\zeta}_j\frac{\partial}{\partial \bar{\zeta}_j}\\
   &=\zeta_0\overline{Z}.
   \end{split}
\end{equation}
Hence, we have
\begin{equation}
    \boldsymbol{L}_{\sum_{j=1}^{n+1}\bar{\zeta}_jX_j}\circ {\boldsymbol{P}}_{w-1,w'}=(w'-k+1){\boldsymbol{P}}_{w-1,w'}.
\end{equation}
By Definition \ref{P}, composing with ${\boldsymbol{M}}_{w-1,w'}$ and ${\boldsymbol{M}}_{k-w-1,k-w'-1}$, yields the final conclusion.
\end{proof}
We now turn to the desired spectral estimate. The fact that $P^{\theta_0}_k$ is formally self-adjoint implies that if $u_t$ is a one-parameter family of functions in $S^{k,2}(S^{2n+1}; \R)$ with $u_0=u$, then
\begin{align}
   \label{variation1}\frac{d}{dt}\bigg|_{t=0}A^{\theta_0}_k(u)&=2\int_{S^{2n+1}}\dot{u}P^{\theta_0}_k (u) d\xi, \\ \label{variation2}
    \frac{d^2}{dt^2}\bigg|_{t=0}A^{\theta_0}_k(u)&=2\int_{S^{2n+1}}\dot{u}P^{\theta_0}_k (\dot{u}) d\xi+2\int_{S^{2n+1}}\ddot{u}P^{\theta_0}_k (u) d\xi. 
\end{align}
for $\dot{u}:=\frac{\partial}{\partial t}\big|_{t=0}u_t$ and $\ddot{u}:=\frac{\partial^2}{\partial t^2}\big|_{t=0}u_t$.
\begin{proof}[Proof of Proposition \ref{spectral}]
Let $v\in C^{\infty}(S^{2n+1};\R)$ be such that
\begin{equation}\label{ccbalanced}
    \int_{S^{2n+1}}v u^{p-1}d\xi=0.
\end{equation}
Then $u_t=\|u+tv\|^{-1}_{p}(u+tv)$ defines a smooth curve with $B^{\theta_0}_k(u_t)=1$ and $u_0=u$, $\dot{u}_t\big|_{t=0}=v$. Since $u$ is a critical point of $A^{\theta_0}_k$, it follows from (\ref{variation1}) that
\begin{displaymath}
    P^{\theta_0}_k(u)=A^{\theta_0}_k(u)u^{p-1}.
\end{displaymath}
Since $u$ is a local minimizer, $\frac{d^2}{dt^2}\bigg|_{t=0}A^{\theta_0}_k(u_t)\geqslant 0$. Expanding this using (\ref{variation2}) and the above display yields
\begin{equation}\label{spectral2}
    \int_{S^{2n+1}}v P^{\theta_0}_k(v)d\xi \geqslant (p-1)A^{\theta_0}_k(u)\int_{S^{2n+1}}u^{p-2}v^2d\xi.
\end{equation}
The assumption (\ref{cbalanced}) implies that for each coordinate function $z_j=x_j+i y_j$, $1\leqslant j\leqslant n+1$, the functions $x_j u$ and $y_j u$ satisfy (\ref{ccbalanced}). Since $P^{\theta_0}_k$ is self-adjoint, it follows from (\ref{spectral2}) that
\begin{displaymath}
\begin{split}
    \sum_{j=1}^{n+1}\int_{S^{2n+1}}\bar{z}_j uP_k^{\theta_0}(z_ju)d\xi&=\sum_{j=1}^{n+1}\left(\int_{S^{2n+1}}x_j uP_k^{\theta_0}(x_ju)d\xi+\int_{S^{2n+1}}y_j uP_k^{\theta_0}(y_ju)d\xi\right)\\
    &\geqslant (p-1)A^{\theta_0}_k(u)\int_{S^{2n+1}}\sum_{j=1}^{n+1}(x_j^2+y_j^2)u^{p}d\xi\\
    &=(p-1)A^{\theta_0}_k(u).
    \end{split}
\end{displaymath}
The final conclusion follows from the definition of $[P_k^{\theta_0},z_j]$.
\end{proof}

Proposition \ref{spectral} and Corollary \ref{transl} reduce the problem of classifying positive local minimizers of $A^{\theta_0}_k$ to the problem of showing that the only functions which satisfy (\ref{spectral3}) are the constants. This can be done for the operators $P_{w,w}$ and $P_{w-1,w}$ using the commutator identity in Theorem \ref{commutator}.
\begin{proof}[Proof of Theorem \ref{class}]
All computations in this proof are carried out with respect to the CR unit sphere $(S^{2n+1}, T^{1,0}S^{2n+1}, \theta_0)$ with the volume element $d\xi$. As discussed in \cite[Theorem 1.3]{FL10}, we may assume that the optimizer $u$ of (\ref{sharpSo}) is a positive real function; i.e. $u$ is a positive real minimizer of 
\begin{displaymath}
A^{\theta_0}_k(u)=\int_{S^{2n+1}}uP^{\theta_0}_k(u)d\xi
\end{displaymath}
with $B^{\theta_0}_k(u)=1$. By Corollary \ref{transl} we may assume that $u$ satisfies (\ref{cbalanced}). We conclude from Proposition \ref{spectral} that
\begin{displaymath}
\sum_{j=1}^{n+1}\int_{S^{2n+1}}\bar{z}_j u[P_k^{\theta_0},z_j](u)d\xi \geqslant (p-2) \int_{S^{2n+1}}uP^{\theta_0}_k(u)d\xi.
\end{displaymath}
Since $P_k^{\theta_0}=P_{w,w}$ with $k=2w+n+1$, combining this with Theorem \ref{commutator} yields
\begin{equation}\label{spe}
0\geqslant \int_{S^{2n+1}}u\left((p-2)P_{w,w}+k(w-k+1)P_{w-1,w}\right)(u)d\xi.
\end{equation}
It follows from the factorization in Theorem \ref{fact} that
\begin{equation}
    P_{w,w}=\prod_{j=0}^{k-1}L_{k-2j-1}, \quad P_{w-1,w}=\prod_{j=0}^{k-2}L_{k-2j-1}
\end{equation}
where $L_{\mu}$ is defined in (\ref{L}). Hence, by simple calculation, we find that
\begin{equation}\label{operator}
(p-2)P_{w,w}+k(w-k+1)P_{w-1,w}=(p-2)\left(\frac{1}{2}\Delta_b+\frac{i}{2}(1-k) T\right)\prod_{j=0}^{k-2}L_{k-2j-1}.
\end{equation}
From (\ref{B}) and (\ref{T}), we have that for $\alpha\in \R$,
\begin{equation}
    \Delta_b+i\alpha T=\frac{n+\alpha}{n}\Box_b+\frac{n-\alpha}{n}\overline{\Box}_b.
\end{equation}
Thus, $|\alpha|<n$ implies that the operator $\Delta_b+i\alpha T$ has trivial kernel because $\Box_b$ and $\overline{\Box}_b$ are nonnegative. Notice that $|k-1|<n$ and $|k-2j-1|<n$ for $j\in \{0,\cdots, k-2\}$. We conclude that the operator in (\ref{operator}) is nonnegative with kernel exactly equal to the constant functions. Combing this with (\ref{spe}) yields $u=1$. The final conclusion follows from the fact that if $\Phi\in {\mathcal{A}}(S^{2n+1})$, then 
\begin{displaymath}
|J_{\Phi}|(\eta)=\frac{C}{|1-\xi \cdot \bar{\eta}|^Q}, \quad \eta\in S^{2n+1}
\end{displaymath}
for some $C>0, \xi\in \C^{n+1}, |\xi|<1$.
\end{proof}

\section{Almost Sobolev inequalities}
In this section, we use Lemma \ref{ccp} to prove Aubin's almost sharp inequalities for GJMS operators on general CR manifolds. This argument is similar to the proof of Theorem \ref{improve}.

\begin{proof}[Proof of Theorem \ref{almost2}:]
Let $\alpha=C_{n,2k}+\epsilon$. If (\ref{Sobolev2}) is not true, then for any $j\in \N$, we can find a $F_j\in S^{k,2}(N)$ satisfying
\begin{equation}
    \left(\int_{N}|F_j|^p d\zeta\right)^{\frac{2}{p}}> \alpha\int_{N}\bar{F}_j \tilde{P}^{\theta}_k F_j d\zeta+j\int_N |F_j|^2 d\zeta.
\end{equation}
We may assume
\begin{displaymath}
    \left(\int_{N}|F_j|^p d\zeta\right)^{\frac{2}{p}}=1.
\end{displaymath}
Then
\begin{equation}
    \int_{N}\bar{F}_j \tilde{P}^{\theta}_k F_j d\zeta<\frac{1}{\alpha}, \quad \int_N |F_j|^2 d\zeta<\frac{1}{j}.
\end{equation}
It follows that $F_j\rightharpoonup 0$ weakly in $S^{k,2}(N)$. After passing to a subsequence, we assume as measures,
\begin{equation}
    \bar{F}_j \tilde{P}^{\theta}_k F_j d\zeta\to \sigma, \quad |F_j|^p d\zeta \to \nu.
\end{equation}
By Lemma \ref{ccp} we can find countably many points $\{x_i\}\in N$ such that
\begin{equation}\label{bubble}
    \nu=\sum_{i}\nu_i \delta_{x_i}, \quad \nu_i^{\frac{2}{p}}\leqslant C_{n,2k} \sigma_i
\end{equation}
here $\nu_i=\nu(x_i)$, $\sigma_i=\sigma(x_i)$. Then
\begin{equation}\label{volume}
    \nu(N)=1,\quad \sigma(N)<\frac{1}{\alpha}.
\end{equation}
Then by the concavity of the function $x^{\frac{2}{p}}$, we have
\begin{equation}
    1\leqslant \sum_i \nu_i^{\frac{2}{p}}\leqslant \sum_i C_{n,2k} \sigma_i=C_{n,2k} \sigma(N)=\frac{C_{n,2k}}{\alpha}.
\end{equation}
Hence,
\begin{displaymath}
    \alpha\leqslant C_{n,2k}.
\end{displaymath}
This contradicts the choice of $\alpha$.
\end{proof}

\section{Existence of minimizers}
Analogous to the CR Yamabe problem \cite{JL87}, we prove that minimizers of $Y_k(N)$ exist if $Y_k(N)<Y_k(\H^n)$.

\begin{proof}[Proof of Theorem \ref{highCR}]
Let $\{F_i\}$ be a minimizing sequence for $Y_k(N)$ with $B^{\theta}_k(F_i)=1$; i.e.
\begin{equation}\label{energy}
    Y_k(N)=\lim_{i\to \infty} \int_N F_i P_k^{\theta}(F_i) \theta\wedge d\theta^n.
\end{equation}
In particular, $\{F_i\}$ is bounded in $S^{k,2}(N;\R)$. By taking a subsequence if necessary, we assume that there exists a limit function $F\in S^{k,2}(N;\R)$ such that
\begin{itemize}
    \item $F_i\rightharpoonup F$ in $S^{k,2}(N;\R)$;
    \item $F_i\to F$ in $L^{2}(N;\R)$; and 
    \item $F_i\to F$ almost everywhere in $N$.
\end{itemize}
Now, Theorem \ref{almost2} implies that for any $\epsilon>0$, there is a constant $C(\epsilon)>0$ such that for all $i$, 
\begin{equation}\label{non}
    \left(\int_{N}|F_i|^p d\zeta\right)^{\frac{2}{p}}\leqslant \left(C_{n,2k}+\epsilon\right)\int_{N} F_i \tilde{P}^{\theta}_k F_i d\zeta+C(\epsilon)\int_N |F_i|^2 d\zeta.
\end{equation}
If the limit function $F$ is zero, by the compact embedding $S^{k,2}(N)\hookrightarrow S^{k',2}(N)$ for $k'<k$, all lower order terms in $\int_{N} F_i {P}^{\theta}_k F_i d\zeta$ will vanish as $i\to \infty$, i.e.
\begin{equation}
    Y_k(N)=\lim_{i\to \infty}\int_{N} F_i {P}^{\theta}_k F_i d\zeta= \lim_{i\to \infty}\int_{N} F_i \tilde{P}^{\theta}_k F_i d\zeta.
\end{equation}
Thus for any $\delta>0$, there is an integer $I$ such that for all $i>I$, we have
\begin{displaymath}
\int_{N} F_i \tilde{P}^{\theta}_k F_i d\zeta<Y_k(N)+\delta.
\end{displaymath}
Combining this with (\ref{non}) yields
\begin{displaymath}
    1\leqslant (Y_k^{-1}(\H^n)+\epsilon)(Y_k(N)+\delta).
\end{displaymath}
Since $Y_k(N)<Y_k(\H^n)$, we may choose $\epsilon$ and $\delta$ such that $(Y_k^{-1}(\H^n)+\epsilon)(Y_k(N)+\delta)<1$, a contradiction.

Finally, we set $G_i=F_i-F$. Then
\begin{itemize}
    \item $G_i\rightharpoonup 0$ in $S^{k,2}(N;\R)$;
    \item $G_i\to 0$ in $L^{2}(N;\R)$; and 
    \item $G_i\to 0$ almost everywhere in $N$.
\end{itemize}
Then Lemma 2.6 in \cite{Lieb83} implies that
\begin{displaymath}
    \|F\|^p_p=1-\lim_{i\to \infty} \|G_i\|^p_p.
\end{displaymath}
In particular, $\|F\|^p_p\in (0,1]$ and $\lim_{i\to \infty} \|G_i\|^p_p\in [0,1)$. We deduce that
\begin{equation}\label{concave}
    \lim_{i\to \infty} \|G_i\|^2_p\geqslant 1-\|F\|^2_p
\end{equation}
with equality if and only if $\|G_i\|_p\to 0$ as $i\to \infty$. Since $G_i\rightharpoonup 0$, it holds that
\begin{displaymath}
\begin{split}
    \lim_{i\to \infty}\int_N F_i {P}_k^{\theta}(F_i) \theta\wedge d\theta^n &=\int_N F {P}_k^{\theta}(F) \theta\wedge d\theta^n+\lim_{i\to \infty}\int_N G_i {P}_k^{\theta}(G_i) \theta\wedge d\theta^n\\
    &=\int_N F {P}_k^{\theta}(F) \theta\wedge d\theta^n+\lim_{i\to \infty}\int_N G_i \tilde{P}_k^{\theta}(G_i) \theta\wedge d\theta^n.
    \end{split}
\end{displaymath}
Therefore
\begin{equation}\label{energy4}
    Y_k(N)=A^{\theta}_k(F)+\lim_{i\to \infty}\int_N G_i \tilde{P}_k^{\theta}(G_i) \theta\wedge d\theta^n.
\end{equation}
From (\ref{concave}) and the definition of $Y_k(N)$, when $Y_k(N)\geqslant 0$, we deduce that
\begin{equation}\label{energy2}
\lim_{i\to \infty}\int_N G_i {P}_k^{\theta}(G_i) \theta\wedge d\theta^n\geqslant Y_k(N)\lim_{i\to \infty} \|G_i\|^2_p\geqslant Y_k(N)(1-\|F\|^2_p).
\end{equation}
When $Y_k(N)<0$, it holds automatically that 
\begin{equation}\label{energy3}
    \lim_{i\to \infty}\int_N G_i {P}_k^{\theta}(G_i) \theta\wedge d\theta^n\geqslant Y_k(N)(1-\|F\|^2_p).
\end{equation}
because $\lim_{i\to \infty}A^{\theta}_k(G_i)$ is nonnegative. Moreover, equalities can be attained in (\ref{energy2}) and (\ref{energy3}) if and only if $\|G_i\|_p\to 0$ as $i\to \infty$. Combining this with (\ref{energy4}) yields
\begin{equation}
    Y_k(N)B^{\theta}_k(F)\geqslant A^{\theta}_k(F).
\end{equation}
It follows from the definition of $Y_k(N)$ that the equality holds. In particular, $A^{\theta}_k(F)=Y_k(N)$ and $B^{\theta}_k(F)=1$ which implies that $F_i\to F$ in $S^{k,2}(N;\R)$ and $F$ is a minimzier of $Y_k(N)$.

As a critical point of $Y_k(N)$, $F$ satisfies
\begin{equation}\label{EL}
    P^{\theta}_kF=Y_k(N)|F|^{p-2}F.
\end{equation}
By Lemma \ref{regularity} below, we obtain that $F\in L^r(N)$ for all $r\geqslant 1$. Then, $P^{\theta}_kF\in L^r(N)$ for all $r\geqslant 1$. It follows from the regularity theorem \cite[Theorem 6.1]{Folland75} that $F\in S^{2k,r}(N)$ for all $r\geqslant 1$, and from the embedding theorem in \cite[Proposition 5.7(a)]{JL87} that $F\in \Gamma_{\beta}(N)$ for all $\beta<2k$. Plugging this result into (\ref{EL}) yields $F\in C^{2k}(N)$.
\end{proof}
\begin{remark}
When $k\geqslant 2$, it is difficult to know whether the minimizer is positive. We refer the readers to \cite{HY14} for some sufficient conditions in the study of $Q$-curvature in conformal geometry. If the minimizer $F$ is positive, iterating the regularity argument in Theorem \ref{highCR} shows that $F$ is smooth. 
\end{remark}

\begin{lemma}\label{regularity}
Let $(N^{2n+1}, T^{1,0}N)$ be a compact strictly pseudoconvex CR manifold with contact form $\theta$. Denote $Q=2n+2$ and $k\in\N$ with $k<n+1$. Let $f$ be a real-valued function defined on $N$ and $u\in S^{k,2}(N;\R)$ be a weak solution of 
\begin{equation}\label{f}
 P^{\theta}_k u=fu.
\end{equation}
If $f\in L^{\frac{Q}{2k}}(N)$, then $u\in L^r(N)$ for all $r\geqslant 1$.
\end{lemma}
\begin{proof}
We proceed as in \cite{DHL00}. As a starting point, we claim that for any $\epsilon>0$, there exists $g_{\epsilon}\in L^{\frac{Q}{2k}}(N)$, an $h_{\epsilon}\in L^{\infty}(N)$, and a constant $C(\epsilon)>0$ such that
\begin{displaymath}
fu=g_{\epsilon}u+h_{\epsilon}, \quad \|g_{\epsilon}\|_{\frac{Q}{2k}}<\epsilon, \quad \|h_{\epsilon}\|\leqslant C(\epsilon).
\end{displaymath}
Here we may assume that $f\neq 0$, and we let
\begin{displaymath}
\begin{split}
    E_l:&=\left\{x\in N; |f|<l\right\},\\
    F_m:&=\left\{x\in N; |u|<m\right\},
\end{split}
\end{displaymath}
where if $\hat{\epsilon}$ is such that $(2\hat{\epsilon})^{\frac{2k}{Q}}=\frac{\epsilon}{2}$, $l$ and $m$ are chosen such that
\begin{displaymath}
    \begin{split}
    \|f\|_{L^{\frac{Q}{2k}}(N\backslash E_l)}<\hat{\epsilon},& \quad \|f\|_{L^{\frac{Q}{2k}}(N\backslash F_m)}<\hat{\epsilon},\\
    E_l\cap F_m\neq \emptyset,& \quad f\big|_{E_l\cap F_m}\neq 0. 
    \end{split}
\end{displaymath}
Given $K\geqslant 1$ an integer that we fix below, we define
\begin{displaymath}
    g_{\epsilon}(x)=\left\{
             \begin{array}{ll}
             \frac{1}{K}f(x), & x\in E_l\cap F_m \\
             f(x), & x\in N\backslash E_l\cap F_m \\
             \end{array}
\right.
\end{displaymath}
and 
\begin{displaymath}
    h_{\epsilon}=(f-g_{\epsilon})u.
\end{displaymath}
Clearly, $h_{\epsilon}=0$ on $N\backslash E_l\cap F_m$. On the other hand,
\begin{displaymath}
    \begin{split}
        \|g_{\epsilon}\|_{\frac{Q}{2k}}^{\frac{Q}{2k}}& =\int_{E_l\cap F_m} |g_{\epsilon}|^{\frac{Q}{2k}}\theta\wedge d\theta^n+ \int_{N\backslash E_l\cap F_m} |g_{\epsilon}|^{\frac{Q}{2k}}\theta\wedge d\theta^n\\
        & \leqslant \int_{E_l\cap F_m} |g_{\epsilon}|^{\frac{Q}{2k}}\theta\wedge d\theta^n+\int_{N\backslash E_l} |g_{\epsilon}|^{\frac{Q}{2k}}\theta\wedge d\theta^n+\int_{N\backslash F_m} |g_{\epsilon}|^{\frac{Q}{2k}}\theta\wedge d\theta^n\\
        & \leqslant \left(\frac{1}{K}\right)^{\frac{Q}{2k}} \int_{E_l\cap F_m} |f|^{\frac{Q}{2k}}\theta\wedge d\theta^n+2\hat{\epsilon},
    \end{split}
\end{displaymath}
so that
\begin{displaymath}
    \|g_{\epsilon}\|_{\frac{Q}{2k}}\leqslant \frac{1}{K}\|f\|_{\frac{Q}{2k}}+\frac{\epsilon}{2}.
\end{displaymath}
Choosing $K$ such that $\frac{1}{K}\|f\|_{\frac{Q}{2k}}<\frac{\epsilon}{2}$, we get
\begin{displaymath}
    \|g_{\epsilon}\|_{\frac{Q}{2k}}<\epsilon.
\end{displaymath}
Now, since $h_{\epsilon}=0$ on $N\backslash E_l\cap F_m$, 
\begin{displaymath}
    \|h_{\epsilon}\|_{\infty}\leqslant \Big|1-\frac{1}{K}\Big|lm,
\end{displaymath}
and this proves the above claim.

Now the equation (\ref{f}) can be written as
\begin{equation}\label{f2}
    P^{\theta}_k u=g_{\epsilon}u+h_{\epsilon}.
\end{equation}
Adding a positive constant if necessary, we may assume that $P^{\theta}_k$ is positive define. Let $\mathcal{H}_{\epsilon}$ be the operator
\begin{displaymath}
    \mathcal{H}_{\epsilon}u=\left(P^{\theta}_k\right)^{-1}\left(g_{\epsilon}u\right).
\end{displaymath}
The equation (\ref{f2}) becomes 
\begin{equation}
    u-\mathcal{H}_{\epsilon}u=\left(P^{\theta}_k\right)^{-1}(h_{\epsilon}).
\end{equation}
For any $r>1$ and any $\tilde{f}\in L^r(N)$, by the regularity theorem \cite[Theorem 6.1]{Folland75}, there exists a unique solution $\tilde{u}\in S^{k,r}(N)$ such that $P^{\theta}_k \tilde{u}=\tilde{f}$. Let $v\in L^r(N)$, $r\geqslant \frac{2Q}{Q-2k}$, and $u_{\epsilon}\in S^{k,r}(N)$ be such that
\begin{displaymath}
    P^{\theta}_k u_{\epsilon}=g_{\epsilon}v.
\end{displaymath}
We set $\hat{r}=\frac{Qr}{Q+2kr}$. By H\"older's inequality, 
\begin{equation}\label{g}
    \|g_{\epsilon}v\|_{\hat{r}} \leqslant \|g_{\epsilon}\|_{\frac{Q}{2k}}\|v\|_r.
\end{equation}
It follows from \cite[Theorem 6.1]{Folland75} and \cite[Theorem 5.5]{JL87} that there exists a constant $C$ such that
\begin{displaymath}
    \|u_{\epsilon}\|_r\leqslant C\|g_{\epsilon}v\|_{\hat{r}}.
\end{displaymath}
Combing this with (\ref{g}), we obtain that 
\begin{displaymath}
    \|u_{\epsilon}\|_r\leqslant C\epsilon \|v\|_{r}.
\end{displaymath}
In other words, for all $r\geqslant \frac{2Q}{Q-2k}$, we have
\begin{displaymath}
         \mathcal{H}_{\epsilon}: L^r(N) \to L^r(N),\quad
        \|\mathcal{H}_{\epsilon}\|_{L^r \to L^r}\leqslant C\epsilon.
\end{displaymath}
Let $r\geqslant \frac{2Q}{Q-2k}$ be given. For $\epsilon>0$ sufficiently small, we obtain that
\begin{displaymath}
    \|\mathcal{H}_{\epsilon}\|_{L^r \to L^r}\leqslant \frac{1}{2}.
\end{displaymath}
This implies that the operator
\begin{displaymath}
    \mathrm{Id}-\mathcal{H}_{\epsilon}: L^r(N) \to L^r(N)
\end{displaymath}
has an inverse. Since
\begin{displaymath}
    \mathrm{Id}-\mathcal{H}_{\epsilon} u=\left(P^{\theta}_k\right)^{-1}(h_{\epsilon})
\end{displaymath}
and $u\in L^{\frac{2Q}{Q-2k}}$, $h_{\epsilon}\in L^{\infty}(N)$, we get that $u\in L^r(N)$.
\end{proof}

\bibliography{mybib}{}

\begin{thebibliography}{HMM17}

\bibitem[Aub76]{Aubin76}
Thierry Aubin.
\newblock Probl\`emes isop\'{e}rim\'{e}triques et espaces de {S}obolev.
\newblock {\em J. Differential Geometry}, 11(4):573--598, 1976.

\bibitem[BFM13]{BFM07}
Thomas~P. Branson, Luigi Fontana, and Carlo Morpurgo.
\newblock Moser-{T}rudinger and {B}eckner-{O}nofri's inequalities on the {CR}
  sphere.
\newblock {\em Ann. of Math. (2)}, 177(1):1--52, 2013.

\bibitem[Cas21]{Case19}
Jeffrey~S. Case.
\newblock The {F}rank-{L}ieb approach to sharp {S}obolev inequalities.
\newblock {\em Commun. Contemp. Math.}, 23(3):Paper No. 2050015, 16, 2021.

\bibitem[CH22]{CH22}
Sun-Yung~A. Chang and Fengbo Hang.
\newblock Improved {M}oser-{T}rudinger-{O}nofri inequality under constraints.
\newblock {\em Comm. Pure Appl. Math.}, 75(1):197--220, 2022.

\bibitem[CMY17]{CMY17}
Jih-Hsin Cheng, Andrea Malchiodi, and Paul Yang.
\newblock A positive mass theorem in three dimensional {C}auchy-{R}iemann
  geometry.
\newblock {\em Adv. Math.}, 308:276--347, 2017.

\bibitem[CMY19]{CMY19}
Jih-Hsin Cheng, Andrea Malchiodi, and Paul Yang.
\newblock On the sobolev quotient of three-dimensional cr manifolds.
\newblock {\em arXiv: Differential Geometry}, 2019.

\bibitem[DHL00]{DHL00}
Zindine Djadli, Emmanuel Hebey, and Michel Ledoux.
\newblock Paneitz-type operators and applications.
\newblock {\em Duke Math. J.}, 104(1):129--169, 2000.

\bibitem[FL12]{FL10}
Rupert~L. Frank and Elliott~H. Lieb.
\newblock Sharp constants in several inequalities on the {H}eisenberg group.
\newblock {\em Ann. of Math. (2)}, 176(1):349--381, 2012.

\bibitem[Fol75]{Folland75}
G.~B. Folland.
\newblock Subelliptic estimates and function spaces on nilpotent {L}ie groups.
\newblock {\em Ark. Mat.}, 13(2):161--207, 1975.

\bibitem[FS74]{FS74}
G.~B. Folland and E.~M. Stein.
\newblock Estimates for the {$\bar \partial _{b}$} complex and analysis on the
  {H}eisenberg group.
\newblock {\em Comm. Pure Appl. Math.}, 27:429--522, 1974.

\bibitem[GG05]{GG05}
A.~Rod Gover and C.~Robin Graham.
\newblock C{R} invariant powers of the sub-{L}aplacian.
\newblock {\em J. Reine Angew. Math.}, 583:1--27, 2005.

\bibitem[HMM17]{Hirachi17}
Kengo Hirachi, Taiji Marugame, and Yoshihiko Matsumoto.
\newblock Variation of total {$Q$}-prime curvature on {CR} manifolds.
\newblock {\em Adv. Math.}, 306:1333--1376, 2017.

\bibitem[HW20]{HW21}
Fengbo Hang and Xiaodong Wang.
\newblock Improved sobolev inequality under constraints.
\newblock {\em arXiv: Classical Analysis and ODEs}, 2020.

\bibitem[HY16]{HY14}
Fengbo Hang and Paul~C. Yang.
\newblock {$Q$}-curvature on a class of manifolds with dimension at least 5.
\newblock {\em Comm. Pure Appl. Math.}, 69(8):1452--1491, 2016.

\bibitem[JL87]{JL87}
David Jerison and John~M. Lee.
\newblock The {Y}amabe problem on {CR} manifolds.
\newblock {\em J. Differential Geom.}, 25(2):167--197, 1987.

\bibitem[Koh65]{Kohn65}
J.~J. Kohn.
\newblock Boundaries of complex manifolds.
\newblock In {\em Proc. {C}onf. {C}omplex {A}nalysis ({M}inneapolis, 1964)},
  pages 81--94. Springer, Berlin, 1965.

\bibitem[Lee86]{Lee86}
John~M. Lee.
\newblock The {F}efferman metric and pseudo-{H}ermitian invariants.
\newblock {\em Trans. Amer. Math. Soc.}, 296(1):411--429, 1986.

\bibitem[Lie83]{Lieb83}
Elliott~H. Lieb.
\newblock Sharp constants in the {H}ardy-{L}ittlewood-{S}obolev and related
  inequalities.
\newblock {\em Ann. of Math. (2)}, 118(2):349--374, 1983.

\bibitem[Lio84]{Lions84}
P.-L. Lions.
\newblock The concentration-compactness principle in the calculus of
  variations. {T}he locally compact case. {I}.
\newblock {\em Ann. Inst. H. Poincar\'{e} Anal. Non Lin\'{e}aire},
  1(2):109--145, 1984.

\bibitem[Lio85]{Lions85}
P.-L. Lions.
\newblock The concentration-compactness principle in the calculus of
  variations. {T}he limit case. {I}.
\newblock {\em Rev. Mat. Iberoamericana}, 1(1):145--201, 1985.

\bibitem[LP87]{LP87}
John~M. Lee and Thomas~H. Parker.
\newblock The {Y}amabe problem.
\newblock {\em Bull. Amer. Math. Soc. (N.S.)}, 17(1):37--91, 1987.

\bibitem[Pon08]{Ponge05}
Rapha\"{e}l~S. Ponge.
\newblock Heisenberg calculus and spectral theory of hypoelliptic operators on
  {H}eisenberg manifolds.
\newblock {\em Mem. Amer. Math. Soc.}, 194(906):viii+ 134, 2008.

\bibitem[Put20]{Putterman20}
Eli Putterman.
\newblock Cubature formulas and sobolev inequalities.
\newblock {\em arXiv: Combinatorics}, 2020.

\bibitem[Tak18]{Takeuchi17}
Yuya Takeuchi.
\newblock Ambient constructions for {S}asakian {$\eta$}-{E}instein manifolds.
\newblock {\em Adv. Math.}, 328:82--111, 2018.

\end{thebibliography}
\bibliographystyle{alpha}

\end{document}